\documentclass[a4paper,11pt]{article}
\usepackage{amsmath,amssymb,amsthm}
\usepackage{graphicx, color}
\usepackage[top=80pt, bottom=90pt, left=70pt, right=70pt]{geometry}
\usepackage{comment}
\usepackage{ascmac}
\usepackage{fancybox}

\newcommand{\average}[1]{\ensuremath{\langle#1\rangle} }
\newtheorem{thm}{\bfseries Theorem}[section]
\newtheorem{lem}[thm]{\bfseries Lemma} 
\newtheorem{prop}[thm]{\bfseries Proposition} 

\newtheorem{cl}[thm]{\bfseries Claim}

\newtheorem{exmp}[thm]{\bfseries Example}

\usepackage{bm}
\allowdisplaybreaks
\begin{document}
\title{On a general framework for network representability\\in discrete optimization\footnotemark[1]}
\author{Yuni Iwamasa\footnotemark[2]}
\footnotetext[1]{A preliminary version of this paper has appeared in the proceedings of the 4th International Symposium on Combinatorial Optimization (ISCO 2016).}
\footnotetext[2]{Department of Mathematical Informatics,
Graduate School of Information Science and Technology,
University of Tokyo, Tokyo, 113-8656, Japan.\\
Email: \texttt{yuni\_iwamasa@mist.i.u-tokyo.ac.jp}
}
\maketitle
\begin{abstract}
	In discrete optimization, representing an objective function as an $s$-$t$ cut function of a network is a basic technique to design an efficient minimization algorithm.
	A network representable function can be minimized by computing a minimum $s$-$t$ cut of a directed network, which is an efficiently solvable problem.
	Hence it is natural to ask what functions are network representable.
	In the case of pseudo Boolean functions (functions on $\{0,1\}^n$), it is known that any submodular function on $\{0,1\}^3$ is network representable.
	\v{Z}ivn{\'y}--Cohen--Jeavons showed by using the theory of expressive power that a certain submodular function on $\{0,1\}^4$ is not network representable.
	
	In this paper, we introduce a general framework for the network representability of functions on $D^n$, where $D$ is an arbitrary finite set.
	We completely characterize network representable functions on $\{0,1\}^n$ in our new definition.
	We can apply the expressive power theory to the network representability in the proposed definition.
	We prove that some ternary bisubmodular function and some binary $k$-submodular function are not network representable.
\end{abstract}
\begin{quote}
\end{quote}

\section{Introduction}
The minimum $s$-$t$ cut problem is one of the most fundamental and efficiently solvable problems in discrete optimization.
Thus, representing a given objective function by the $s$-$t$ cut function of some network leads to an efficient minimization algorithm.
This idea goes back to a classical paper by Iv\u{a}nescu~\cite{OR/I65} in 60's,
and revived in the context of computer vision in the late 80's.
Efficient image denoising and other segmentation algorithms are designed via representing the energy functions as $s$-$t$ cut functions.
Such a technique {\it (Graph Cut)} is now popular in computer vision; see~\cite{JRSSSB/GPS89,TPAMI/KZ04} and references therein.
Also an $s$-$t$ cut function is a representative example of {\it submodular functions}.
Mathematical modeling and learning algorithms utilizing submodularity are now intensively studied in the literature of machine learning; see e.g.~\cite{FTML/B13}.
Hence efficient minimization algorithms of submodular functions are of great importance,
but it is practically impossible to minimize very large submodular functions in machine learning by using generic polynomial time submodular minimization algorithms such as~\cite{book/GLS88,JACM/IFF01,FOCS/LSW15,JCTSB/S00}.
Thus, understanding efficiently minimizable subclasses of submodular functions and developing effective uses of these subclasses for practical problems are important issues.

What (submodular) functions are efficiently minimizable via a network representation and minimum cut computation?
Iv\u{a}nescu~\cite{OR/I65} showed that all submodular functions on $\{0,1\}^2$ are network representable,
and Billionnet--Minoux~\cite{DAM/BM85} showed that the same holds for all submodular functions on $\{0,1\}^3$.
It is meaningful to investigate network representability of functions having few variables,
since they can be used as building blocks for large network representations.
Kolmogorov--Zabih~\cite{TPAMI/KZ04} introduced a formal definition of the network representability,
and showed that network representable functions are necessarily submodular.
Are all submodular functions network representable?
This question was negatively answered by \v{Z}ivn{\'y}--Cohen--Jeavons~\cite{DAM/ZCJ09}.
They showed that a certain submodular function on $\{0,1\}^4$ is {\it not} network representable.
In proving the non-existence of a network representation, they utilized the theory of {\it expressive power} developed in the context of {\it valued constraint satisfaction problems}.

In this paper, we initiate a network representation theory for functions on $D^n$, where $D$ is a general finite set beyond $\{0,1\}$.
Our primary motivation is to give a theoretical basis for applying network flow methods to multilabel assignments, such as the Potts model.
Our main target as well as our starting point is network representations of $k$-submodular functions.
{\it $k$-submodular functions}~\cite{ISCO/HK12} have recently gained attention as a promising generalization of submodular functions on $\{0,1,2,\dots,k\}^n$~\cite{ICCV/GK13,DO/H15,SICOMP/IWY16}.
Iwata--Wahlstr{\"o}m--Yoshida~\cite{SICOMP/IWY16} considered a network representation of $k$-submodular functions for the design of FPT algorithms.
Independently, Ishii~\cite{master/I14_E} considered another representation,
and showed that all 2-submodular (bisubmodular) functions on $\{0,-1,1\}^2$ are network representable.
In this paper, by generalizing and abstracting their approaches,
we present a unified framework for network representations of functions on $D^n$.
Features of the proposed framework as well as results of this paper are summarized as follows:

\begin{itemize}
	\item In our network representation, to represent a function on $D^n$, each variable in $D$ is associated with several nodes.
	More specifically, three parameters $(k, \rho, \sigma)$ define one network representation.
	The previous network representations (by Kolmogorov--Zabih, Ishii, and Iwata--Wahlstr{\"o}m--Yoshida)
	can be viewed as our representations for special parameters.
	\item We completely characterize network representable functions on $\{0,1\}^n$ under our new definition;
	they are network representable in the previous sense or they are monotone (Theorems~\ref{thm:{0,1}main} and~\ref{thm:{0,1}eq};
	reformulated as Theorems~\ref{thm:{0,1}mainExpressive} and~\ref{thm:{0,1}eqExpressive}).
	The minimization problem of monotone functions is trivial.
	This means that it is sufficient only to consider the original network representability for functions on $\{0,1\}^n$.
	\item Our framework is compatible with the expressive power theory,
	which allows us to prove that a function cannot admit any network representation.
	\item
	As an application of the above,
	we provide a bisubmodular function on $\{0,-1,1\}^3$ and a $k$-submodular function on $\{0,1,2,\dots,k\}^2$ for any $k \geq 3$ which are {\it not} network representable for a natural parameter (Theorems~\ref{thm:3ary2sub} and~\ref{thm:2aryksub}; strengthened to Theorems~\ref{thm:3ary2subExpressive} and~\ref{thm:2aryksubExpressive}).
	This answers negatively an open problem raised by~\cite{SICOMP/IWY16}.
\end{itemize}

\paragraph{Organization.}
In Section~\ref{sec:preliminaries}, we introduce submodular functions, $s$-$t$ cut functions, and $k$-submodular functions.
We also introduce the network representation of submodular functions by Kolmogorov--Zabih~\cite{TPAMI/KZ04}.
Furthermore we explain concepts of expressive power and weighted polymorphisms,
which play key roles in proving the non-existence of a network representation.
In Section~\ref{sec:general framework}, we explain the previous network representations of $k$-submodular functions.
Then we introduce a framework for the network representability of functions on $D^n$,
and discuss its compatibility with the expressive power theory.
We also present our results on network representability in our framework.
In Section~\ref{sec:remarks}, we present several remarks about {\it (submodular) representability} and {\it extended pp-interpretation}.
In Section~\ref{sec:proof}, we give proofs of statements in Section~\ref{sec:general framework}.

\paragraph{Notation.}
Let $\mathbf{Q}$ and $\mathbf{Q}_+$ denote the sets of rationals and nonnegative rationals, respectively.
In this paper, functions can take the infinite value $+ \infty$, where $a < + \infty$ and $a + \infty = + \infty$ for $a \in \mathbf{Q}$.
Let $\overline{\mathbf{Q}} := \mathbf{Q} \cup \{+ \infty\}$.
For a function $f : D^n \rightarrow \overline{\mathbf{Q}}$, let ${\rm dom}\ f := \{ x \in D^n \mid f(x) < + \infty \}$.
For a positive integer $k$, let $[k] := \{1,2,\dots,k\}$, and $[0, k] := [k] \cup \{0\}$.
By a (directed) network $(V,A;c)$, we mean a directed graph $(V,A)$ endowed with rational nonnegative edge capacity $c : A \rightarrow \mathbf{Q}_+ \cup \{+\infty\}$.
A subset $X \subseteq V$ is also regarded as a characteristic function $X : V \rightarrow \{0,1\}$ defined by $X(i) := 1$ for $i \in X$ and $X(i) := 0$ for $i \not\in X$.
A function $\rho : F \rightarrow E$ with $F \supseteq E$ is called a {\it retraction} if it satisfies $\rho(a) = a$ for $a \in E$.
$\rho : F \rightarrow E$ is extended to $\rho : F^n \rightarrow E^n$ by defining $(\rho(x))_i := \rho(x_i)$ for $x \in F^n$ and $i \in [n]$.

\section{Preliminaries}\label{sec:preliminaries}
\subsection{Submodularity}
A {\it submodular function} is a function $f$ on $\{0,1\}^n$ satisfying the following inequalities
\begin{align*}
f(x) + f(y) \geq f(x \wedge y) + f(x \vee y) \qquad (x,y \in \{0,1\}^n),
\end{align*}
where binary operations $\wedge, \vee$ are defined by
\begin{align*}
(x \wedge y)_i := \begin{cases} 1 & \text{if $x_i = y_i = 1$,} \\ 0 & \text{if $x_i = 0$ or $y_i = 0$}, \end{cases} \qquad (x \vee y)_i := \begin{cases} 1 & \text{if $x_i = 1$ or $y_i = 1$,} \\ 0 & \text{if $x_i = y_i = 0$}, \end{cases}
\end{align*}
for $x = (x_1, x_2, \dots, x_n)$ and $y = (y_1, y_2, \dots, y_n)$.

The {\it $s$-$t$ cut function} of a network $G = (V \cup \{s,t\}, A; c)$ is a function $C$ on $2^V$ defined by
\begin{align*}
C(X) := \sum_{(u,v) \in A,\ u \in X \cup \{s\},\ v \not\in X \cup \{s\}} c(u,v) \qquad (X \subseteq V).
\end{align*}
For $X \subseteq V$, we call $X \cup \{s\}$ an {\it $s$-$t$ cut}.
An $s$-$t$ cut function is submodular.
In particular, an $s$-$t$ cut function can be efficiently minimized by a max-flow min-cut algorithm.
The current fastest one is $O(|V||A|)$-time algorithm by Orlin~\cite{STOC/O13}.

Let us introduce a class of functions on $[0,k]^n$, which also plays key roles in discrete optimization.
A {\it $k$-submodular function} is a function $f$ on $[0,k]^n$ satisfying the following inequalities
\begin{align*}
f(x) + f(y) \geq f(x \sqcap y) + f(x \sqcup y) \qquad (x, y \in [0,k]^n),
\end{align*}
where binary operations $\sqcap, \sqcup$ are defined by
\begin{align*}
(x \sqcap y)_i := \begin{cases}
x_i & \text{if $x_i = y_i$}, \\
0 & \text{if $x_i \neq y_i$},
\end{cases}
\quad (x \sqcup y)_i := \begin{cases}
\text{$y_i$ (resp. $x_i$)} & \text{if $x_i = 0$ (resp. $y_i = 0$)},\\
x_i & \text{if $x_i = y_i$}, \\
0 & \text{if $0 \neq x_i \neq y_i \neq 0$},
\end{cases}
\end{align*}
for $x = (x_1, x_2, \dots, x_n)$ and $y = (y_1, y_2, \dots, y_n)$.
A $k$-submodular function was introduced by Huber--Kolmogorov~\cite{ISCO/HK12} as an extension of submodular functions.
In $k = 1$, a $k$-submodular function is submodular,
and in $k = 2$, a $k$-submodular function is called {\it bisubmodular}, which domain is typically written as $\{0, -1, 1\}^n$ (see~\cite{book/Fujishige05}).
It is not known whether a $k$-submodular function can be minimized in polynomial time under the value oracle model for $k \geq 3$.
By contrast, Thapper--\v{Z}ivn\'y~\cite{FOCS/TZ12} proved that $k$-submodular functions can be minimized in polynomial time in the valued constraint satisfaction problem model for all $k$ (see~\cite{SICOMP/KTZ15} for the journal version).

In the following, we denote the set of all submodular functions having at most $n$ variables as $\Gamma_{{\rm sub,}n}$,
and let $\Gamma_{\rm sub} := \bigcup_n \Gamma_{{\rm sub,}n}$.
We also denote the set of all bisubmodular functions (resp. $k$-submodular functions) having at most $n$ variables as $\Gamma_{{\rm bisub,}n}$ (resp. $\Gamma_{{k{\rm sub,}}n}$).

\subsection{Network representation over $\{0,1\}$}
A function $f : \{0,1\}^n \rightarrow \overline{\mathbf{Q}}$ is said to be {\it network representable} if there exist a network $G = (V, A; c)$ and a constant $\kappa \in \mathbf{Q}$ satisfying the following:
\begin{itemize}
	\item $V \supseteq \{s,t,1,2,\dots, n\}$.
	\item For all $x = (x_1,x_2,\dots, x_n) \in \{0,1\}^n$, it holds that
	\begin{align*}
	f(x) = \min\{C(X) \mid \text{$X$: $s$-$t$ cut, } X(i) = x_i \text{ for } i \in [n] \} + \kappa.
	\end{align*}
\end{itemize}
This definition of the network representability was introduced by Kolmogorov--Zabih~\cite{TPAMI/KZ04}.
A network representable function has the following useful properties:
\begin{description}
	\item[Property 1:] A network representable function $f$ can be minimized via computing a minimum $s$-$t$ cut of a network representing $f$.
	\item[Property 2:] The sum of network representable functions $f_1, f_2$ is also network representable,
	and a network representation of $f_1+f_2$ can easily be constructed by combining networks representing $f_1, f_2$.
\end{description}
By the property 1, a network representable function can be minimized efficiently, provided a network representation is given.
By the property 2, it is easy to construct a network representation of a function $f$ if $f$ is the sum of ``smaller'' network representable functions.
Hence it is meaningful to investigate network representability of functions having few variables.
For example, by the fact that all submodular functions on $\{0,1\}^2$ are network representable,
we know that the sum of submodular functions on $\{0,1\}^2$ is also network representable.
This fact is particularly useful in computer vision applications.
Moreover, thanks to extra nodes,
a function obtained by a {\it partial minimization} (defined in Section~\ref{subsec:expressive}) of a network representable function is also network representable.

\subsection{Expressive power}\label{subsec:expressive}
It turned out that the above definition of network representability is suitably dealt with in the theory of {\it expressive power}, which has been developed in the literature of valued constraint satisfaction problems~\cite{book/Zivny12}.
The term ``expressive power'' has been used for various different meanings.
In this paper, ``expressive power'' is meant as a class of functions closed under several operations,
which is formally introduced as follows.

Let $D$ be a finite set, called a {\em domain}.
A {\em cost function} on $D$ is a function $f : D^r \rightarrow \overline{\mathbf{Q}}$ for some positive integer $r = r_f$, called the {\em arity} of $f$.
A set of cost functions on $D$ is called a {\em language} on $D$.
A cost function $f_{=} : D^2 \rightarrow \overline{\mathbf{Q}}$ defined by
$f_{=}(x,y) := 0$ if $x = y$ and $f_{=}(x,y) := +\infty$ if $x \neq y$,
is called the {\it weighted equality relation} on $D$.
A {\it weighted relational clone}~\cite{SICOMP/CCCJZ13} on $D$ is a language $\Gamma$ on $D$ such that
\begin{itemize}
	\item $f_{=} \in \Gamma$,
	\item for $\alpha \in \mathbf{Q}_+$, $\beta \in \mathbf{Q}$, and $f \in \Gamma$, it holds that $\alpha f + \beta \in \Gamma$,
	\item any addition of $f, g \in \Gamma$ belongs to $\Gamma$, and
	\item for $f \in \Gamma$, any partial minimization of $f$ belongs to $\Gamma$.
\end{itemize}
Here an {\it addition} of two cost functions $f, g$ is a cost function $h$ obtained by
\begin{align*}
h(x_1,\dots,x_n) = f(x_{s_1(1)}, \dots, x_{s_1(r_f)}) + g(x_{s_2(1)}, \dots, x_{s_2(r_g)}) \quad (x_1,\dots,x_n \in D)
\end{align*}
for some $s_1 : [r_f] \rightarrow [n]$ and $s_2 : [r_g] \rightarrow [n]$.
A {\it partial minimization} of $f$ of arity $n+m$ is a cost function $h$ of arity $n$ obtained by
\begin{align*}
h(x_1,\dots,x_n) = \min_{x_{n+1},\dots,x_{n+m} \in D}f(x_1,\dots,x_n, x_{n+1},\dots,x_{n+m}) \quad (x_1,\dots,x_n \in D).
\end{align*}
For a language $\Gamma$, the {\it expressive power} $\average{\Gamma}$ of $\Gamma$ is the smallest weighted relational clone (as a set) containing $\Gamma$~\cite{book/Zivny12}.
A cost function $f$ is said to be {\it representable by a language $\Gamma$}
if $f \in \average{\Gamma}$.

By using these notions, \v{Z}ivn{\'y}--Cohen--Jeavons~\cite{DAM/ZCJ09} noted that the set of network representable functions are equal to the expressive power of $\Gamma_{\rm sub,2}$.
\begin{lem}[\cite{DAM/ZCJ09}]\label{lem:Gamma_sub2}
	The set of network representable functions coincides with $\average{\Gamma_{\rm sub,2}}$.
\end{lem}
The previous results for network (non)representability are summarized as follows.
\begin{thm}
	The following hold:
	\begin{description}
		\item[\cite{TPAMI/KZ04}] $\average{\Gamma_{\rm sub,2}} \subseteq \Gamma_{\rm sub}$.
		\item[\cite{DAM/BM85}] $\average{\Gamma_{\rm sub,2}} = \average{\Gamma_{\rm sub,3}}$.
		\item[\cite{DAM/ZCJ09}] $\average{\Gamma_{\rm sub,2}} \not\supseteq \Gamma_{\rm sub,4}$.
	\end{description}
\end{thm}
When proving $\average{\Gamma_{\rm sub,2}} \not\supseteq \Gamma_{\rm sub,4}$,
\v{Z}ivn{\'y}--Cohen--Jeavons~\cite{DAM/ZCJ09} actually found a 4-ary submodular function $f$ such that $f \not\in \average{\Gamma_{\rm sub,2}}$.

\subsection{Weighted polymorphisms}\label{subsec:wPol}
How can we prove $f \not\in \average{\Gamma}$?
We here introduce algebraic objects known as {\it weighted polymorphisms}, for proving this.
A function $\varphi : D^k \rightarrow D$ is called a $k$-ary {\it operation} on $D$.
For $x^1 = (x^1_1,x^1_2,\dots, x^1_n),\dots,x^k = (x^k_1,x^k_2,\dots, x^k_n) \in D^n$, we define $\varphi(x^1,x^2,\dots,x^k)$ by $(x^1,x^2,\dots,x^k) \mapsto \left(\varphi(x^1_1,x^2_1,\dots,x^k_1), \dots, \varphi(x^1_n,x^2_n, \dots, x^k_n)\right) \in D^n$.
A $k$-ary {\it projection} $e^{(k)}_i$ for $i \in [k]$ on $D$ is defined by
$x \mapsto x_i$ for $x = (x_1,x_2,\dots,x_k) \in D^k$.
A $k$-ary operation $\varphi$ is called a {\it polymorphism} of $\Gamma$ if for all $f \in \Gamma$ and for all $x^1,x^2,\dots,x^k \in {\rm dom}\ f$, it holds that $\varphi(x^1,x^2,\dots,x^k) \in {\rm dom}\ f$.
Let ${\rm Pol}^{(k)}(\Gamma)$ be the set of $k$-ary polymorphisms of $\Gamma$,
and let ${\rm Pol}(\Gamma) := \bigcup_k {\rm Pol}^{(k)}(\Gamma)$.
Note that for any $\Gamma$, all projections are in ${\rm Pol}(\Gamma)$.
Let us define a weighted polymorphism.
A function $\omega : {\rm Pol}^{(k)}(\Gamma) \rightarrow \mathbf{Q}$ is called a $k$-ary {\it weighted polymorphism} of $\Gamma$~\cite{SICOMP/CCCJZ13} if it satisfies the following:
\begin{itemize}
	\item $\sum_{\varphi \in {\rm Pol}^{(k)}(\Gamma)} \omega(\varphi) = 0$.
	\item If $\omega(\varphi) < 0$, then $\varphi$ is a projection.
	\item For all $f \in \Gamma$ and for all $x^1,x^2, \dots, x^k \in {\rm dom}\ f$,
	\begin{align*}
	\sum_{\varphi \in {\rm Pol}^{(k)}(\Gamma)} \omega(\varphi) f(\varphi(x^1,x^2,\dots, x^k)) \leq 0.
	\end{align*}
\end{itemize}
Let ${\rm wPol}^{(k)}(\Gamma)$ be the set of $k$-ary weighted polymorphisms of $\Gamma$,
and let ${\rm wPol}(\Gamma) := \bigcup_k {\rm wPol}^{(k)}(\Gamma)$.
Here the following lemma holds:
\begin{lem}[\cite{SICOMP/CCCJZ13}]\label{lem:exp chara}
	Suppose that $\Gamma$ is a language on $D$
	and $f$ is a cost function on $D$.
	If there exist some $\omega \in {\rm wPol}^{(k)}(\Gamma)$ and $x^1,x^2, \dots, x^k \in {\rm dom}\ f$ satisfying
	\begin{align*}
	\sum_{\varphi \in {\rm Pol}^{(k)}(\Gamma)} \omega(\varphi) f(\varphi(x^1,x^2,\dots,x^k)) > 0,
	\end{align*}
	then it holds that $f \not\in \average{\Gamma}$.
\end{lem}
Thus we can prove nonrepresentability by using Lemma~\ref{lem:exp chara}.

\section{General framework for network representability}\label{sec:general framework}
\subsection{Previous approaches of network representation over $D$}\label{subsec:previous}
Here we explain previous approaches of network representation for functions over a general finite set $D$.
Ishii~\cite{master/I14_E} considered a method of representing a bisubmodular function, which is a function on $\{0, -1, 1\}^n$, by a skew-symmetric network.
A network $G = (\{s^+, s^-, 1^+,1^-, \dots, N^+, N^-\}, A; c)$ is said to be {\it skew-symmetric} if it satisfies that if $(u,v) \in A$, then $(\overline{v}, \overline{u}) \in A$ and $c(u,v) = c(\overline{v}, \overline{u})$.
Here define $\overline{u}$ by $\overline{u} := i^+$ if $u = i^-$ and $\overline{u} := i^-$ if $u = i^+$.
An $s^+$-$s^-$ cut $X$ is said to be {\it transversal} if $X \not\supseteq \{i^+, i^-\}$ for every $i \in [n]$.
The set of transversal $s^+$-$s^-$ cuts is identified with $\{0,-1,1\}^N$ by $X \mapsto x_i := X(i^+) - X(i^-)$ for $i \in [N]$.
Ishii gave a definition of the network representability for a function on $\{0, -1, 1\}^n$ as follows:
\begin{quote}
	A function $f : \{0, -1, 1\}^n \rightarrow \overline{\mathbf{Q}}$ is said to be {\it skew-symmetric network representable} if there exist a skew-symmetric network $G = (V,A; c)$ and a constant $\kappa \in \mathbf{Q}$ satisfying the following:
	\begin{itemize}
		\item $V \supseteq \{s^+, s^-, 1^+,1^-, 2^+,2^-, \dots, n^+, n^- \}$.
		\item For all $x = (x_1,x_2,\dots, x_n) \in \{0, -1, 1\}^n$,
		\begin{align*}
		f(x) = \min\{ C(X) \mid \text{$X$: transversal $s^+$-$s^-$ cut, }X(i^+)-X(i^-) = x_i \text{ for } i \in [n] \} + \kappa.
		\end{align*}
	\end{itemize}
\end{quote}
In a skew-symmetric network,
the minimal minimum $s^+$-$s^-$ cut is transversal~\cite{master/I14_E}.
Hence a skew-symmetric network representable function can be minimized efficiently via computing a minimum $s^+$-$s^-$ cut.
Here the following holds:
\begin{lem}[\cite{master/I14_E}]\label{lem:ishii}
	Skew-symmetric network representable functions are bisubmodular.
\end{lem}
Moreover Ishii proved the following theorem:
\begin{thm}[\cite{master/I14_E}]\label{thm:ishii}
	All binary bisubmodular functions are skew-symmetric network representable.
\end{thm}
This representation has both Property 1 and Property 2. 
Therefore a bisubmodular function given as the sum of binary bisubmodular functions is skew-symmetric network representable.
Thanks to extra nodes,
a bisubmodular function given as partial minimization of a skew-symmetric network representable function is also skew-symmetric network representable.

Iwata--Wahlstr{\"o}m--Yoshida~\cite{SICOMP/IWY16} considered another method of representing a $k$-submodular function by a network.
\begin{quote}
	A function $f : [0,k]^n \rightarrow \overline{\mathbf{Q}}$ is said to be {\it $k$-network representable} if there exist a network $G = (V,A; c)$ and a constant $\kappa \in \mathbf{Q}$ satisfying the following:
	\begin{itemize}
		\item $V = \{s, t\} \cup \{ i^l \mid (i,l) \in [n] \times [k]\}$\footnote[2]{The symbol $i^l$ is \emph{not} meant as a number $\underbrace{i \times i \times \cdots \times i}_{l} \in \mathbf{Z}$.}.
		\item The $s$-$t$ cut function $C$ of $G$ satisfies
		\begin{align*}
		C(X) \geq C(\underline{X}) \qquad (X: \text{$s$-$t$ cut}),
		\end{align*}
		where $\underline{X} := \{s\} \cup \bigcup_{i \in [n]} \{ i^l \mid X \cap \{ i^1,i^2,\dots,i^k \} = i^l \text{ for some $l \in [k]$} \}$.
		\item For all $x = (x_1,x_2,\dots, x_n) \in [0,k]^n$, it holds that
		\begin{align*}
		f(x) = C(X_{x}) + \kappa,
		\end{align*}
		where $X_{x} := \{s\} \cup \bigcup_{x_i \neq 0} \{i^l \mid x_i = l\}$.
	\end{itemize}
\end{quote}
$k$-network representable functions can be minimized via computing a minimum $s$-$t$ cut by definition,
and constitute an efficiently minimizable subclass of $k$-submodular functions, as follows.
\begin{lem}[\cite{SICOMP/IWY16}]\label{lem:IWY}
	$k$-network representable functions are $k$-submodular.
\end{lem}
Iwata--Wahlstr{\"o}m--Yoshida constructed networks representing {\it basic $k$-submodular functions}, which are special $k$-submodular functions.
This method also has both Property~1 and Property~2.
Therefore a $k$-submodular function given as the sum of basic $k$-submodular functions is $k$-network representable.

As seen in Section~\ref{subsec:expressive}, network representable functions on $\{0,1\}^n$ are considered as the expressive power of $\Gamma_{\rm sub,2}$,
and hence we can apply the expressive power theory to network representability.
However Ishii and Iwata--Wahlstr{\"o}m--Yoshida network representation methods cannot enjoy the expressive power theory by the following reasons:
\begin{description}
	\item[(i)] The set of network representable functions under Iwata--Wahlstr{\"o}m--Yoshida method is not a weighted relational clone,
	since their method does not allow the existence of extra nodes.
	\item[(ii)] The concept of expressive power only focuses on the representability of functions on the same domain,
	while Ishii and Iwata--Wahlstr{\"o}m--Yoshida methods consider representations of functions over $[0,k]$ by functions over $\{0,1\}$.
\end{description}
We introduce, in the next subsection, a new network represetability definition for resolving (i),
and in Section~\ref{subsec:VCSP perspective}, we also introduce an extension of expressive power for resolving (ii).

\subsection{Definition}\label{subsec:def}
By abstracting the previous approaches,
we here develop a unified framework for network representability over $D$.
The basic idea is the following:
Consider networks having nodes $i^1,i^2,\dots, i^k$ for each $i \in [n]$,
where $|D| \leq 2^k$.
We associate one variable $x_i$ over $D$ with $k$ nodes $i^1, i^2,\dots, i^k$.
The $k$ nodes have $2^k$ intersection patterns with $s$-$t$ cuts.
We specify a set of $|D|$ patterns,
which represents $D$, for each $i$.
The cut function restricted to cuts with specified patterns gives a function on $D^n$.
To remove effect of irrelevant patterns in minimization,
we fix a retraction from all patterns to specified patterns,
and consider networks with the property that the retraction does not increase cut capacity.
Now functions represented by such networks are minimizable via minimum $s$-$t$ cut with retraction.

A formal definition is given as follows.
Let $k$ be a positive integer, and $E$ a subset of $\{0,1\}^k$.
We consider a node $i^l$ for each $(i, l) \in [n] \times [k]$.
For a retraction $\rho : \{0,1\}^k \rightarrow E$, a network $G = (V, A; c)$ is said to be {\it $(n,\rho)$-retractable} if $G$ satisfies the following:
\begin{itemize}
	\item $V \supseteq \{s, t\} \cup \{ i^l \mid (i,l) \in [n] \times [k]\}$.
	\item For all $x = (x_1^1,\dots, x_1^k, x_2^1,\dots,x_2^k,\dots, x_n^1,\dots,x_n^k) \in \{0,1\}^{kn}$,
	\begin{align*}
	C_{\rm min}(x) \geq C_{\rm min}(\rho(x_1^1,\dots,x_1^k), \dots, \rho(x_n^1,\dots,x_n^k)),
	\end{align*}
	where
	\begin{align*}
	C_{\rm min}(x) := \min \{ C(X) \mid \text{$X$: $s$-$t$ cut, } X(i^l) = x_i^l \text{ for } (i,l) \in [n] \times [k]\}.
	\end{align*}
\end{itemize}
Let $\sigma$ be a bijection from $D$ to $E$.
A function $f : D^n \rightarrow \overline{\mathbf{Q}}$ is said to be {\it $(k, \rho, \sigma)$-network representable} if there exist an $(n, \rho)$-retractable network $G = (V, A; c)$ and a constant $\kappa \in \mathbf{Q}$ satisfying that
\begin{align*}
f(x) = C_{\rm min}(\sigma(x_1), \sigma(x_2), \dots, \sigma(x_n)) + \kappa
\end{align*}
for all $x = (x_1,x_2,\dots, x_n) \in D^n$.
A $(k,\rho,\sigma)$-network representable function can be minimized efficiently via computing a minimum $s$-$t$ cut.

\begin{exmp}
\upshape
Let $\rho_2 : \{0,1\}^2 \rightarrow \{(0,0), (0,1), (1,0) \}$ be a retraction defined by $\rho_2(x) := x$ if $x \in \{(0,1), (1,0)\}$ and $\rho_2(x) := (0,0)$ if $x \in \{(0,0), (1,1)\}$.
Suppose that all edge capacities of a network are finite.
Since a network with $2n$ nodes is $(n, \rho_2)$-retractable if and only if the vector of edge capacities satisfies some linear inequalities,
the set of $(n, \rho_2)$-retractable networks with $2n$ nodes forms a polyhedral cone.
Hence every $(n, \rho_2)$-retractable network with $2n$ nodes can be represented as a nonnegative combination of extreme rays of the cone.
Fig.~\ref{fig:2_rho_2sub_network} illustrates all types of extremal $(2, \rho_2)$-retractable networks with four nodes,
where each network is a representative of equivalence class induced by $+ \leftrightarrow -$ and $1 \leftrightarrow 2$.
We obtained these networks via a computer calculation.

Every skew-symmetric network can be represented as a synthesis of the three figures in Fig.~\ref{fig:2_rho_2sub_network} by the definition; first, fourth, and fifth from the left in the first row.
Indeed, for any skew-symmetric network $G = (\{s^+,s^-,1^+,1^-, \dots, n^+,n^-\}, A; c)$ and any distinct $i,j \in [n]$,
the subgraph of $G$ induced by $\{s^+, s^-, i^+, i^-, j^+, j^- \}$ is represented as a nonnegative combination of the four networks.
Thus, $G$ is an $(m,\rho_2)$-retractable network for all $m \leq n$.
\begin{figure}[htb]
\begin{center}
\includegraphics[clip,width=8.0cm]{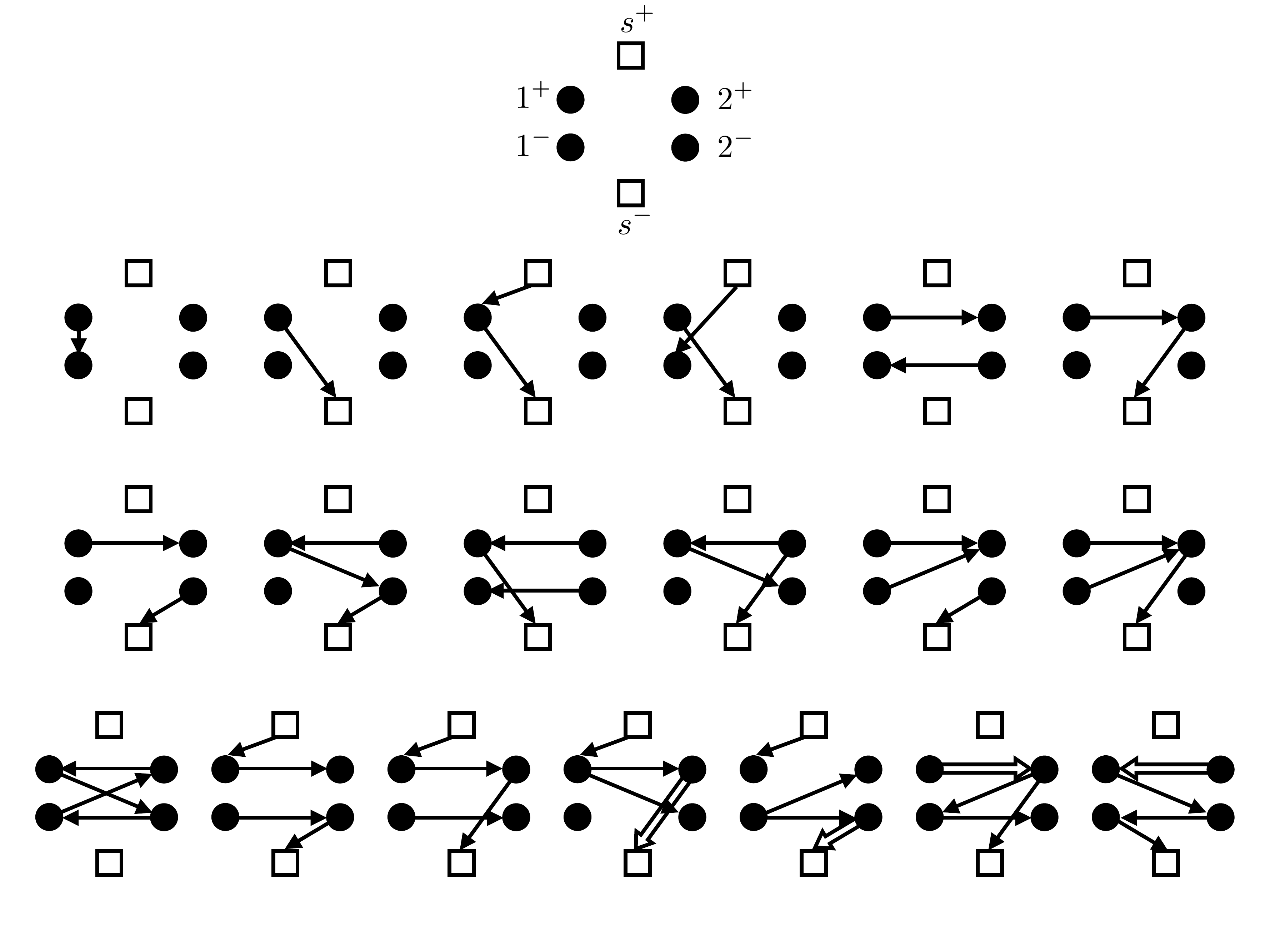}
\caption{All types of extremal $(2, \rho_2)$-retractable networks with four nodes and finite edge capacities,
where each network is a representative of an equivalence class induced by $+ \leftrightarrow -$ and $1 \leftrightarrow 2$. Single edge means the capacity of the edge is equal to 1, and double edge means the capacity of the edge is equal to 2.}
\label{fig:2_rho_2sub_network}
\end{center}
\end{figure}
\end{exmp}

The network representability in the sense of Kolmogorov--Zabih is the same as the $(1,{\rm id}, {\rm id})$-network representability,
where ${\rm id} : \{0,1\} \rightarrow \{0,1\}$ is the identity map.
Let $\sigma_k : [0,k] \rightarrow \{0, 1\}^k$ and $\rho_k : \{0, 1\}^k \rightarrow \{0,1\}^k$ be maps defined by
\begin{align*}
&\sigma_k(x) := \begin{cases}
(0,\dots,0,\smash{\overset{i}{\check{1}}},0,\dots,0) & \text{if $x = i \in [k]$},\\
(0, \dots, 0) & \text{if $x = 0$},
\end{cases}\\
&\rho_k(x) := \begin{cases}
x & \text{if $x = (0,\dots,0, \smash{\overset{i}{\check{1}}},0,\dots,0)$ for some $i \in [k]$},\\
(0, \dots, 0) & \text{otherwise}.
\end{cases}
\end{align*}
Then the skew-symmetric network representability is a special class of the $(2, \rho_2, \sigma_2)$-network representability,
and the $k$-network representability is a special class of the
$(k, \rho_k, \sigma_k)$-network representability.

The $(k,\rho,\sigma)$-network representability possesses both Property 1 and Property 2.
Furthermore a function given as a partial minimization of a $(k,\rho,\sigma)$-network representable function is also $(k,\rho,\sigma)$-network representable.

\subsection{Results on network representability}\label{subsec: results}
In our network representation, one variable is associated with ``several'' nodes even if $D = \{0,1\}$.
Hence the set of network representable functions on $\{0,1\}^n$ in our sense may be strictly larger than that in the original.
The following theorem says that additional network representable functions are only monotone.
\begin{thm}\label{thm:{0,1}main}
	If a function $f$ on $\{0,1\}^n$ is $(k, \rho, \sigma)$-network representable for some $k, \rho, \sigma$,
	then $f$ is $(1,{\rm id}, {\rm id})$-network representable, or monotone.
	Moreover some monotone function is not $(k, \rho, \sigma)$-network representable for any $k, \rho, \sigma$.
\end{thm}
The minimization of a monotone function is trivial.
Therefore it is sufficient only to consider $(1, {\rm id}, {\rm id})$-network representability (original network representability) for functions on $\{0,1\}^n$.
	Here note that the sum of $(1,{\rm id}, {\rm id})$-network representable function $f_1$ and monotone function $f_2$ is not always network representable for some $k, \rho, \sigma$,
since $f_1$ and $f_2$ might use different $k, \rho, \sigma$ for a representation.

We give a more precise structure of network representable functions on $\{0,1\}^n$.
Let $\sigma^*_1 : \{0,1\} \rightarrow \{0,1\}^2$, $\sigma^*_2 : \{0,1\} \rightarrow \{0,1\}^2$, and $\rho^* : \{0,1\}^2 \rightarrow \{0,1\}^2$ be maps defined by
\begin{align*}
\sigma^*_1(x) =
\begin{cases}
(1, 0) & \text{if $x = 1$}, \\
(0, 1) & \text{if $x = 0$},
\end{cases}
\quad
\sigma^*_2(x) =
\begin{cases}
(0,1) & \text{if $x = 1$}, \\
(1,0) & \text{if $x = 0$},
\end{cases}
\quad
\rho^*(x) =
\begin{cases}
(1, 0) & \text{if $x = (1, 0)$}, \\
(0, 1) & \text{otherwise}.
\end{cases}
\end{align*}
Then the following holds:
\begin{thm}\label{thm:{0,1}eq}
	A function $f$ on $\{0,1\}^n$ is $(k,\rho, \sigma)$-network representable for some $k, \rho, \sigma$ if and only if
	$f$ is $(1, {\rm id}, {\rm id})$-network representable, $(2, \rho^*, \sigma^*_1)$-network representable, or $(2, \rho^*, \sigma^*_2)$-network representable.
\end{thm}

We next present network nonrepresentability results for functions on $D^n$,
and in particular, $k$-submodular functions.
These results will be proved via the theory of expressive power.
We have seen in Theorem~\ref{thm:ishii} that all binary bisubmodular functions are $(2, \rho_2, \sigma_2)$-network representable.
We show that the same property does not hold for ternary bisubmodular functions.
\begin{thm}\label{thm:3ary2sub}
	Some ternary bisubmodular function is not $(2, \rho_2, \sigma_2)$-network representable.
\end{thm}
We also know that all binary basic $k$-submodular functions are $(k, \rho_k, \sigma_k)$-network representable~\cite{SICOMP/IWY16},
and their sum is efficiently minimizable.
A natural question raised by~\cite{SICOMP/IWY16} is whether all binary $k$-submodular functions are $k$-network representable or not.
We answer this question negatively.
\begin{thm}\label{thm:2aryksub}
	Some binary $k$-submodular function is not $(k, \rho_k, \sigma_k)$-network representable for all $k \geq 3$.
\end{thm}
Theorems~\ref{thm:{0,1}main}, \ref{thm:{0,1}eq}, \ref{thm:3ary2sub}, and \ref{thm:2aryksub} are consequences of Theorems~\ref{thm:{0,1}mainExpressive}, \ref{thm:{0,1}eqExpressive}, \ref{thm:3ary2subExpressive}, and \ref{thm:2aryksubExpressive} in the next subsection.

\subsection{Extended expressive power}\label{subsec:VCSP perspective}
In order to incorporate the theory of expressive power into our framework, we introduce a way of handling languages on $D$ from a language $\Gamma$ on another  domain $F$,
which generalizes previous arguments.
Let $k$ be a positive integer with $|D| \leq |F|^k$.
Let $E$ be a subset of $F^k$ with $|E| = |D|$,
$\rho : F^k \rightarrow E$ a retraction,
and $\sigma : D \rightarrow E$ a bijection.
We define $\average{\Gamma}^k$ by
\begin{align*}
\average{\Gamma}^k := \{ f \in \average{\Gamma} \mid \text{The arity $r_f$ of $f$ is a multiple of $k$}\}.
\end{align*}
Regard $\average{\Gamma}^k$ as a language on $F^k$; recall that $k$ is the arity of $\rho$.
A function $f$ is {\it representable by $(\Gamma, \rho, \sigma)$} if there exists $g \in \average{\Gamma}^k$ satisfying $g(\rho(v)) \leq g(v)$ for all $v \in \textrm{dom }g$ and
\begin{align*}
 f(x_1,x_2,\dots,x_n) = g(\sigma(x_1), \sigma(x_2), \dots,\sigma(x_n)) \qquad (x_1,x_2,\dots,x_n \in D).
\end{align*}
We define a language $\average{\Gamma}^k_{(\rho, \sigma)}$ on $D$ as the set of functions representable by $(\Gamma, \rho, \sigma)$.

By comparing these notions to our network representations,
we obtain a generalization of Lemma~\ref{lem:Gamma_sub2}.
\begin{lem}
	The set of $(k,\rho,\sigma)$-network representable functions coincides with $\average{\Gamma_{{\rm sub},2}}^k_{(\rho, \sigma)}$.
\end{lem}

The following theorem enables us to deal with our network representability on $D^n$ from the theory of expressive power.
\begin{thm}\label{thm:ExtendExpressive}
	For a language $\Gamma$ on $F$,
	$\average{\Gamma}^k_{(\rho, \sigma)}$ is a weighted relational clone on $D$.
\end{thm}
The proof of Theorem~\ref{thm:ExtendExpressive} is given in Section~\ref{subsec:thm:ExtendExpressive}.

Let $\Gamma$ be a language on $F$.
A function $f$ on $D$ is called {\it representable by $\Gamma$} if $f \in \average{\Gamma}^k_{(\rho, \sigma)}$ for some positive integer $k$, $\rho : F^k \rightarrow E$, and $\sigma : D \rightarrow E$.
The set of cost functions on $D$ representable by a language $\Gamma$ is denoted by $\overline{\average{\Gamma}}_D$.
Notice that $\overline{\average{\Gamma}}_D$ is not a weighted rational clone in general.
By using these notations, Theorems~\ref{thm:{0,1}main} and \ref{thm:{0,1}eq} are reformulated as follows,
since $\average{\Gamma_{\rm sub,2}}$ (resp. $\average{\Gamma_{\rm sub,2}}^k_{(\rho, \sigma)}$) is equal to the set of $(1,{\rm id}, {\rm id})$-network representable (resp. $(k,\rho, \sigma)$-network representable) functions.
Define $\Gamma_{\rm mono}$ as the set of monotone functions over $\{0,1\}$.
\begin{thm}\label{thm:{0,1}mainExpressive}
	$\average{\Gamma_{\rm sub,2}} \subsetneq \overline{\average{\Gamma_{\rm sub,2}}}_{\{0,1\}} \subsetneq \average{\Gamma_{\rm sub,2}} \cup \Gamma_{\rm mono}$.
\end{thm}
\begin{thm}\label{thm:{0,1}eqExpressive}
	$\overline{\average{\Gamma_{{\rm sub},2}}}_{\{0,1\}} = \average{\Gamma_{{\rm sub},2}} \cup \average{\Gamma_{{\rm sub},2}}^2_{(\rho^*, \sigma^*_1)} \cup \average{\Gamma_{{\rm sub},2}}^2_{(\rho^*, \sigma^*_2)}$.
\end{thm}

Theorem~\ref{thm:3ary2sub} is rephrased as $\Gamma_{\rm bisub,3} \not\subseteq \average{\Gamma_{\rm sub,2}}^2_{(\rho_2, \sigma_2)}$,
since $\average{\Gamma_{\rm sub,2}}^2_{(\rho_2, \sigma_2)}$ is equal to the set of $(2, \rho_2, \sigma_2)$-network representable functions.
We prove a stronger statement such that $\Gamma_{\rm bisub,3}$ is not included even in the set of $(\Gamma_{\rm sub}, \rho_2, \sigma_2)$-representable functions.
\begin{thm}\label{thm:3ary2subExpressive}
	$\Gamma_{\rm bisub,3} \not\subseteq \average{\Gamma_{\rm sub}}^2_{(\rho_2, \sigma_2)}$.
\end{thm}
Theorem~\ref{thm:2aryksub} is rephrased as $\Gamma_{k{\rm sub,2}} \not\subseteq \average{\Gamma_{\rm sub,2}}^k_{(\rho_k, \sigma_k)}$,
since $\average{\Gamma_{\rm sub,2}}^k_{(\rho_k, \sigma_k)}$ is equal to the set of $(k, \rho_k, \sigma_k)$-network representable functions.
Again we prove a stronger statement such that $\Gamma_{k{\rm sub,2}}$ is not included even in the set of $(\Gamma_{\rm sub}, \rho_k, \sigma_k)$-representable functions.
\begin{thm}\label{thm:2aryksubExpressive}
	$\Gamma_{k{\rm sub,2}} \not\subseteq \average{\Gamma_{\rm sub}}^k_{(\rho_k, \sigma_k)}$ for all $k \geq 3$.
\end{thm}

The proofs of Theorems~\ref{thm:{0,1}mainExpressive}, \ref{thm:{0,1}eqExpressive}, \ref{thm:3ary2subExpressive}, and \ref{thm:2aryksubExpressive}
are given in Sections~\ref{subsec:thm:{0,1}mainExpressive}, \ref{subsec:thm:{0,1}eqExpressive}, \ref{subsec:thm:3ary2subExpressive}, and \ref{subsec:thm:2aryksubExpressive}, respectively.

\section{Discussion}\label{sec:remarks}
\subsection{Submodular representability}
Theorems~\ref{thm:3ary2subExpressive} and \ref{thm:2aryksubExpressive} suggest function classes represented by submodular functions instead of networks.
A function $f : D^n \rightarrow \overline{\mathbf{Q}}$ is said to be {\it $(k, \rho, \sigma)$-submodular representable} if $f \in \average{\Gamma_{\rm sub}}^k_{(\rho, \sigma)}$ for a positive integer $k$, a retraction $\rho : \{0,1\}^k \rightarrow E$, and a bijection $\sigma : D \rightarrow E$.

It is easier to analyze the submodular representability than the network representability.
Indeed, in the submodular representability, we do not need to consider extra variables by the property $\average{\Gamma_\textrm{sub}} = \Gamma_\textrm{sub}$.
Therefore, an $n$-ary function $f$ is $(k, \rho, \sigma)$-submodular representable if and only if there exists a $kn$-ary submodular function $g$ satisfying $g(\sigma(x_1), \sigma(x_2), \dots, \sigma(x_n)) = f(x_1,x_2,\dots,x_n)$ for any $x_1,x_2,\dots,x_n \in D$ and $g(\rho(v)) \leq g(v)$ for any $v \in \{0,1\}^{kn}$.
The latter condition can be verified by checking the nonemptiness of a polyhedron in $\mathbf{R}^{2^{kn}}$ defined by $O(2^{2kn})$ inequalities.
Thus, the following holds:
\begin{prop}
	For $f : D^n \rightarrow \overline{\mathbf{Q}}$, a bijection $\sigma : D \rightarrow E$, and a retraction $\rho : \{0,1\}^k \rightarrow E$,
	we can determine whether $f \in \average{\Gamma_{\rm sub}}^k_{(\rho, \sigma)}$ in time polynomial of $2^{kn}$.
\end{prop}
In the network representability,
the best known upper bound of the number of extra variables to be added is the $kn$-th Dedekind number $M(kn)$~\cite{DAM/RRLT17}, and $M(kn) \geq 2^{kn \choose \lfloor kn/2 \rfloor}$ holds.

\subsection{Another representation of $k$-submodular functions}
There is another natural parameter $(2k, \tilde{\rho}_{k}, \tilde{\sigma}_{k})$, different from $(k, \rho_k, \sigma_k)$, that represents $k$-submodular functions.
Let  $\tilde{\sigma}_{k} : [0,k] \rightarrow \{0,1\}^{2k}$ and $\tilde{\rho}_{k} : \{0,1\}^{2k} \rightarrow \{0,1\}^{2k}$ be maps defined by
\begin{align*}
&\tilde{\sigma}_{k}(x) :=
\begin{cases}
(\underbrace{0,\dots,0,\smash{\overset{i}{\check{1}}},0,\dots,0}_{k}, \underbrace{1,\dots,1,\smash{\overset{i}{\check{0}}},1,\dots,1}_k) & \text{if $x = i \in [k]$},\\
(0, \dots,0) & \text{if $x = 0$},
\end{cases}\\
&\tilde{\rho}_{k}(x) :=
\begin{cases}
x & \text{if $x = (\underbrace{0,\dots,0,\smash{\overset{i}{\check{1}}},0,\dots,0}_{k}, \underbrace{1,\dots,1,\smash{\overset{i}{\check{0}}},1,\dots,1}_k)$ for some $i \in [k]$},\\
(0,\dots,0) & \text{otherwise}.
\end{cases}
\end{align*}
\begin{lem}
	$(2k, \tilde{\rho}_{k}, \tilde{\sigma}_{k})$-submodular representable functions are $k$-submodular.
\end{lem}
\begin{proof}
	It follows from $\tilde{\rho}_k(\tilde{\sigma}_k(x) \wedge \tilde{\sigma}_k(y)) = \tilde{\sigma}_k(x \sqcap y)$ and $\tilde{\rho}_k(\tilde{\sigma}_k(x) \vee \tilde{\sigma}_k(y)) = \tilde{\sigma}_k(x \sqcup y)$ for all $x, y \in [0,k]$.
\end{proof}
Ishii~\cite{master/I14_E} considered a class of skew-symmetric networks which are $(n, \tilde{\rho}_k)$-retractable,
and discussed the corresponding $(2k, \tilde{\rho}_{k}, \tilde{\sigma}_{k})$-network representable $k$-submodular functions.
This network representation was implicitly considered for $k$-submodular functions arising from minimum multiflow problems in \cite{MP/K94}; see~\cite{DO/H15}.
We raise a question:
{\it How are the $(k,\rho_k, \sigma_k)$- and $(2k, \tilde{\rho}_k, \tilde{\sigma}_k)$-submodular representaions related?}

\subsection{Submodular functions on $k$-diamonds}
The {\it $k$-diamond} is a lattice $D_k := \{\bot, \top, 1,2,\dots, k\}$,
where partial order $\preceq$ is defined by $\bot \prec i \prec \top$ for each $i \in [k]$ and incomparable for distinct $i,j \in [k]$.
A {\it $k$-diamond submodular function}~\cite{RIMS/FKMTT15, DO/K11} is a function $f$ on ${D_k}^n$ satisfying the following inequalities
\begin{align*}
f(x) + f(y) \geq f(x \wedge y) + f(x \vee y) \qquad (x,y \in {D_k}^n),
\end{align*}
where $\wedge$ (resp. $\vee$) is the meet (resp. join) operator on ${D_k}^n$.
Since $D_k$ is not a distributive lattice,
a $k$-diamond submodular function is essentially different from a submodular function on $\{0,1\}^n$.
A polynomial time algorithm for minimizing $k$-diamond submodular functions was discovered, just recently, by Fujishige et al.~\cite{RIMS/FKMTT15}.
The algorithm involves the ellipsoid method,
and is far from practical use.

It would be worth considering $k$-diamond submodular functions that fall into the ordinary submodularity via our framework.
Let $\sigma_\textrm{$k$-dia} : D_k \rightarrow \{0,1\}^k$ and $\rho_\textrm{$k$-dia} : \{0,1\}^k \rightarrow \{0,1\}^k$ be maps defined by
\begin{align*}
&\sigma_\textrm{$k$-dia}(x) := \begin{cases}
(1,1,\dots,1) & \text{$x = \top$},\\
(0,\dots,0,\overset{i}{\check{1}},0,\dots,0) & \text{if $x = i \in [k]$},\\
(0,0,\dots,0) & \text{if $x = \bot$},
\end{cases}\\
&\rho_\textrm{$k$-dia}(x) := \begin{cases}
x & \text{if $x = (0,0,\dots,0)$ or $x = (0,\dots,0,\smash{\overset{i}{\check{1}}},0,\dots,0)$ for some $i \in [k]$},\\
(1,1,\dots,1) & \text{otherwise}.
\end{cases}
\end{align*}
The parameter $(k, \rho_{k{\rm \mathchar`-dia}}, \sigma_{k{\rm \mathchar`-dia}})$ actually defines a class of $k$-diamond submodular functions as follows.
\begin{lem}
	$(k, \rho_{k{\rm \mathchar`-dia}}, \sigma_{k{\rm \mathchar`-dia}})$-submodular representable functions are $k$-diamond submodular.
\end{lem}
\begin{proof}
	It follows from $\rho_\textrm{$k$-dia}(\sigma_\textrm{$k$-dia}(x) \wedge \sigma_\textrm{$k$-dia}(y)) = \sigma_\textrm{$k$-dia}(x \wedge y)$ and $\rho_\textrm{$k$-dia}(\sigma_\textrm{$k$-dia}(x) \vee \sigma_\textrm{$k$-dia}(y)) = \sigma_\textrm{$k$-dia}(x \vee y)$ for all $x, y \in D_k$.
\end{proof}

A canonical example of a binary $k$-diamond submodular function is the distance function $d: {D_k}^2 \rightarrow \mathbf{Q}$ of the Hasse diagram of $D_k$:
\begin{align*}
d(x,y) := \begin{cases}
0 & \text{if $x = y$},\\
1 & \text{if $\{x, y\} = \{\bot, i\}$ or $\{\top, i\}$ for some $i \in [k]$},\\
2 & \text{otherwise}.
\end{cases}
\end{align*}
One can verify that $d$ is actually $k$-diamond submodular; see~\cite[Theorem 3.6]{MPSA/H16} for a general version.
A motivation behind $d$ comes from the {\em minimum $(2,k)$-metric problem ($\textrm{MIN}_{2,k}$)}~\cite{COMB/K98}, which is one of basic problems in facility location and multiflow theory.
The problem $\textrm{MIN}_{2,k}$ asks to minimize a nonnegative sum of $d(v,x_i)$ and $d(x_i,x_j)$ over $x = (x_1,x_2,\ldots,x_n) \in {D_k}^n$.
A combinatorial strongly polynomial time algorithm to solve $\textrm{MIN}_{2,k}$ is currently not known.
If $d \in \average{\Gamma_{\rm sub}}^k_{(\rho_{k{\rm \mathchar`-dia}}, \sigma_{k{\rm \mathchar`-dia}})}$, then $\textrm{MIN}_{2,k}$ could at least be solved by combinatorial strongly polynomial time submodular function minimization algorithms~\cite{JACM/IFF01, JCTSB/S00}.
However we verified by computer calculation:
\begin{prop}
	$d \not\in \average{\Gamma_{\rm sub}}^k_{(\rho_{k{\rm \mathchar`-dia}}, \sigma_{k{\rm \mathchar`-dia}})}$.
\end{prop}
Although this attempt failed, we hope that the reduction idea considered in this section will grow up to be a useful tool of algorithm design.

\subsection{Extended primitive positive interpretation}
One of the referees pointed out a similarity between extended expressive power and {\it primitive positive interpretation} ({\it pp-interpretation}, for short).
Here pp-interpretation is a well-known concept in constraint satisfaction problems,
and its generalization for valued constraint satisfaction problems is defined as follows (see e.g.,~\cite[Definition~5.3]{SICOMP/TZ17}).
Let $\Gamma_D$ and $\Gamma_F$ be languages on $D$ and on $F$, respectively.
Let $k$ be a positive integer with $|D| \leq |F|^k$, $E$ a subset of $F^k$ with $|D| \leq |E|$,
and  $\theta : E \rightarrow D$ a surjective map.
We say that $\Gamma_D$ has a {\it pp-interpretation} in $\Gamma_F$ with parameters $(k, E, \theta)$
if $\average{\Gamma_F}$ contains the following weighted relations:
\begin{itemize}
	\item $\delta_E : F^k \rightarrow \overline{\mathbf{Q}}$ defined by $\delta_E(x) := 0$ for $x \in E$ and $\delta_E(x) := +\infty$ for $x \not\in E$,
	\item $\theta^{-1}(f_{=})$, and
	\item $\theta^{-1}(f)$ for any $f \in \Gamma_D$,
\end{itemize}
where, for $f : D^n \rightarrow \overline{\mathbf{Q}}$, $\theta^{-1}(f) : F^{kn} \rightarrow \overline{\mathbf{Q}}$ is a function satisfying $\theta^{-1}(f)(x_1, x_2, \dots, x_n) = f(\theta(x_1), \theta(x_2), \dots, \theta(x_n))$ for every $x_1, x_2, \dots, x_n \in E$.

It seems that extended expressive power and pp-interpretation cannot be compared with each other,
i.e., one is not a special case of the other.
Here we introduce the concept of {\it extended pp-interpretation}, which generalize both extended expressive power and pp-interpretation.
We define $\Gamma_D$, $\Gamma_F$, $k$, and $E$ as above.
Let $\tilde{E}$ be a subset of $E$ with $|\tilde{E}| = |D|$,
$\rho : E \rightarrow \tilde{E}$ a retraction,
and $\sigma : \tilde{E} \rightarrow D$ a bijection.
We say that $\Gamma_D$ has an {\it extended pp-interpretation} in $\Gamma_F$ with parameters $(k, E, \tilde{E}, \rho, \sigma)$
if $\average{\Gamma_F}$ contains the following weighted relations:
\begin{itemize}
	\item $\delta_E : F^k \rightarrow \overline{\mathbf{Q}}$ defined by $\delta_E(x) := 0$ for $x \in E$ and $\delta_E(x) := +\infty$ for $x \not\in E$,
	\item $(\sigma \circ \rho)^{-1}(f_{=})$, and
	\item $(\sigma \circ \rho)^{-1}(f)$ for any $f \in \Gamma_D$,
\end{itemize}
where, for $f : D^n \rightarrow \overline{\mathbf{Q}}$, $(\sigma \circ \rho)^{-1}(f) : F^{kn} \rightarrow \overline{\mathbf{Q}}$ is a function satisfying $(\sigma \circ \rho)^{-1}(f)(x_1, x_2, \dots, x_n) \geq (\sigma \circ \rho)^{-1}(f)(\rho(x_1), \rho(x_2), \dots, \rho(x_n)) = f(\sigma(\rho(x_1)), \sigma(\rho(x_2)), \dots, \sigma(\rho(x_n)))$ for every $x_1, x_2, \dots, x_n \in E$.
We see that the minimization of the sum of cost functions in $\Gamma_D$ can reduce to the minimization of the sum of corresponding cost functions in $\average{\Gamma_F}$.

The pp-interpretation is captured by {\it weighted varieties}, introduced by Kozik--Ochremiak~\cite{ICALP/KO15}.
We do not know whether the extended expressive power and the extended pp-interpretation can be captured by weighted varieties.
This might be interesting future work.

\section{Proofs}\label{sec:proof}
\subsection{Proof of Theorem~\ref{thm:ExtendExpressive}}\label{subsec:thm:ExtendExpressive}
Let us prove that $\average{\Gamma}^k_{(\rho, \sigma)}$ contains the weighted equality relation on $D$,
and is closed under nonnegative scaling and addition of constants, addition, and partial minimization.
By the definition of expressive power, 
$\average{\Gamma}$ contains the weighted equality relation $g_{=}$ on $F$.
Let $h_{=}(u,v) := g_{=}(u_1,v_1)+g_=(u_2,v_2) \cdots + g_=(u_k,v_k)$ for $u = (u_1,u_2,\dots,u_k), v = (v_1,v_2,\dots,v_k) \in F^k$.
Here it is clear that $h_= \in \average{\Gamma}^k$, $h_=$ is the weighted equality relation on $F^k$, and $h_{=}(\rho(u), \rho(v)) = h_{=}(u,v) = 0$ for $(u,v) \in \textrm{dom }h_=$.
Let $f_= : D^2 \rightarrow \overline{\mathbf{Q}}$ be a cost function defined by $f_=(x,y) := h_{=}(\sigma(x), \sigma(y))$ for $(x,y) \in D^2$.
Then $f_= \in \average{\Gamma}^k_{(\rho, \sigma)}$ and $f_=$ is the weighted equality relation on $D$.

The fact that $\average{\Gamma}^k_{(\rho, \sigma)}$ is closed under nonnegative scaling and addition of constants is trivial.

Let $h : D^n \rightarrow \overline{\mathbf{Q}}$ be a cost function defined by for some $s_1 : [r_f] \rightarrow [n]$ and $s_2 : [r_g] \rightarrow [n]$,
\begin{align*}
h(x_1,x_2,\dots,x_n) := f(x_{s_1(1)},x_{s_1(2)}, \dots, x_{s_1(r_f)}) + g(x_{s_2(1)}, x_{s_2(2)}, \dots, x_{s_2(r_g)})
\end{align*}
for all $x_1,x_2,\dots,x_n \in D$,
where $f, g \in \average{\Gamma}^k_{(\rho, \sigma)}$.
Since $f, g \in \average{\Gamma}^k_{(\rho, \sigma)}$,
there exist $f', g' \in \average{\Gamma}^k$ satisfying $f'(\rho(v)) \leq f'(v)$ for $v \in \textrm{dom $f'$}$, $g'(\rho(v)) \leq g'(v)$ for $v \in \textrm{dom $g'$}$, and
\begin{align*}
f(x_1,x_2, \dots, x_{r_f}) &= f'(\sigma(x_1), \sigma(x_2),\dots, \sigma(x_{r_f}))\qquad (x_1,x_2,\dots,x_{r_f} \in D),\\
g(x_1,x_2, \dots, x_{r_g}) &= g'(\sigma(x_1),\sigma(x_2),\dots, \sigma(x_{r_g}))\qquad (x_1,x_2,\dots,x_{r_g} \in D).
\end{align*}
We define $h' : (F^k)^n \rightarrow \overline{\mathbf{Q}}$ by
\begin{align*}
h'(v_1,v_2,\dots,v_n) := f'(v_{s_1(1)},v_{s_1(2)}, \dots, v_{s_1(r_f)}) + g'(v_{s_2(1)},v_{s_2(2)}, \dots, v_{s_2(r_g)})
\end{align*}
for all $v_1,v_2,\dots,v_n \in F^k$.
Then we have $h(x_1,x_2,\dots,x_n) = h'(\sigma(x_1),\sigma(x_2), \dots, \sigma(x_n))$.
Since $f', g' \in \average{\Gamma}^k$, we obtain $h' \in \average{\Gamma}^k$.
Furthermore it holds that $h'(\rho(v)) \leq h'(v)$ for $v \in \textrm{dom $h'$}$.
That is, $h \in \average{\Gamma}^k_{(\rho, \sigma)}$,
and we know that $\average{\Gamma}^k_{(\rho, \sigma)}$ is closed under addition.

Let $h : D^n \rightarrow \overline{\mathbf{Q}}$ be defined by
\begin{align*}
h(x_1,\dots,x_n) := \min_{x_{n+1},\dots,x_{n+m} \in D}f(x_1,\dots,x_n, x_{n+1},\dots,x_{n+m})
\end{align*}
for all $x_1,\dots,x_n \in D$,
where $f \in \average{\Gamma}^k_{(\rho, \sigma)}$.
Since $f \in \average{\Gamma}^k_{(\rho, \sigma)}$,
there exists $f' \in \average{\Gamma}^k$ satisfying $f'(\rho(v)) \leq f'(v)$ for $v \in \textrm{dom $f'$}$ and
\begin{align*}
f(x_1,x_2, \dots, x_{n+m}) = f'(\sigma(x_1),\sigma(x_2),\dots, \sigma(x_{n+m}))\qquad (x_1,x_2,\dots,x_{n+m} \in D).
\end{align*}
Here we define $h' : (F^k)^n \rightarrow \overline{\mathbf{Q}}$ by
\begin{align*}
h'(v_1,\dots,v_n) := \min_{v_{n+1},\dots,v_{n+m} \in F^k} f'(v_1,\dots,v_n, v_{n+1}, \dots, v_{n+m})
\end{align*}
for all $v_1,v_2,\dots, v_n \in F^k$.
Since $f' \in \average{\Gamma}^k$, we obtain $h' \in \average{\Gamma}^k$.
Furthermore it holds that $h'(\rho(v)) \leq h'(v)$.
Indeed,
\begin{align*}
h'(v_1,\dots,v_n) &= \min_{v_{n+1}, \dots, v_{n+m} \in F^k} h(v_1,\dots,v_n,v_{n+1}, \dots, v_{n+m})\\
&\geq \min_{v_{n+1}, \dots, v_{n+m} \in F^k} h(\rho(v_1),\dots,\rho(v_n),\rho(v_{n+1}), \dots, \rho(v_{n+m}))\\
&= \min_{v_{n+1}, \dots, v_{n+m} \in F^k} h(\rho(v_1),\dots,\rho(v_n),v_{n+1}, \dots, v_{n+m})\\
&= h'(\rho(v_1), \dots, \rho(v_n)).
\end{align*}
Also we have
\begin{align*}
h(x_1,\dots,x_n) &= \min_{x_{n+1},\dots,x_{n+m} \in D}f(x_1,\dots,x_n, x_{n+1},\dots,x_{n+m})\\
&= \min_{x_{n+1},\dots,x_{n+m} \in D}f'(\sigma(x_1),\dots,\sigma(x_n), \sigma(x_{n+1}), \dots, \sigma(x_{n+m}))\\
&= \min_{v_{n+1},\dots,v_{n+m} \in F^k}f'(\sigma(x_1),\dots,\sigma(x_n), v_{n+1}, \dots, v_{n+m})\\
&= h'(\sigma(x_1), \dots, \sigma(x_n)).
\end{align*}
Hence $h \in \average{\Gamma}^k_{(\rho, \sigma)}$,
and we know that $\average{\Gamma}^k_{(\rho, \sigma)}$ is closed under partial minimization.

\subsection{Proof of Theorem~\ref{thm:{0,1}mainExpressive}}\label{subsec:thm:{0,1}mainExpressive}
\begin{lem}\label{lem:first subsetneq}
$\average{\Gamma_{\rm sub,2}} \subsetneq \overline{\average{\Gamma_{\rm sub,2}}}_{\{0,1\}}$.
\end{lem}
\begin{proof}
It is obvious that $\average{\Gamma_{\rm sub,2}} \subseteq \overline{\average{\Gamma_{{\rm sub},2}}}_{\{0,1\}}$.
Let $f : \{0, 1\}^2 \rightarrow \mathbf{Q}$ be a function defined by $f(1,1) := 1$ and $f(x,y) := 0$ for other $(x,y)$.
Since $0 = f(0,1) + f(1,0) < f(0,0) + f(1,1) = 1$, $f$ is not submodular.
Then $f \not\in \average{\Gamma_{{\rm sub},2}}$.
However it holds that $f \in \average{\Gamma_{{\rm sub},2}}^2_{(\rho^*, \sigma^*_1)}$, where $\rho^*$ and $\sigma^*_1$ are defined in Section~\ref{subsec: results}.
Indeed, the network $G = (V, A; c)$ represents $f$,
where $V = \{s,t,1^1,1^2,2^1,2^2\}$, $A = \{(1^1, 2^2), (2^1, 1^2)\}$, and $c(1^1, 2^2) = c(2^1, 1^2) = 1/2$.
\end{proof}

For $x \in \{0,1\}^n$, let $\overline{x} \in \{0,1\}^n$ be $\overline{x}_i := 1-x_i$ for $i \in [n]$.
For $f : \{0,1\}^n \rightarrow \overline{\mathbf{Q}}$, let $\overline{f} : \{0,1\}^n \rightarrow \overline{\mathbf{Q}}$ be a function defined by $x \mapsto f(\overline{x})$.
\begin{lem}\label{lem:f and overline f}
	Suppose that $f$ is a function on $\{0,1\}^n$.
	$f$ is $(\Gamma, \rho, \sigma)$-representable if and only if
	$\overline{f}$ is $(\Gamma, \rho, \overline{\sigma})$-representable,
	where $\overline{\sigma}$ is defined by
	$\overline{\sigma}(x) := \sigma(\overline{x})$ for $x \in \{0,1\}$.
\end{lem}
\begin{proof}
	Suppose that $f \in \average{\Gamma}^k_{(\rho, \sigma)}$,
	where $\rho : \{0,1\}^k \rightarrow \{0,1\}^k$ and $\sigma : \{0,1\} \rightarrow \{0,1\}^k$.
	Then there exists $g \in \average{\Gamma}^k$ satisfying $g(\rho(v)) \leq g(v)$ for $v \in \textrm{dom }g$ and
	\begin{align*}
		f(x_1,x_2,\dots,x_n) = g(\sigma(x_1),\sigma(x_2), \dots, \sigma(x_n))\qquad (x_1,x_2,\dots,x_n \in \{0,1\}).
	\end{align*}
Hence we have
\begin{align*}
	\overline{f}(x_1,x_2,\dots,x_n) &= f(\overline{x_1}, \overline{x_2}, \dots, \overline{x_n})\\
	&= g(\sigma(\overline{x_1}), \sigma(\overline{x_2}), \dots, \sigma(\overline{x_n}))\\
	&= g(\overline{\sigma}(x_1), \overline{\sigma}(x_2), \dots, \overline{\sigma}(x_n))
\end{align*}
for $(x_1,x_2,\dots,x_n) \in \{0,1\}^n$.
This means $\overline{f} \in \average{\Gamma}^k_{(\rho, \overline{\sigma})}$.
\end{proof}

\begin{lem}\label{lem:network rep}
	Suppose that $f$ is a function on $\{0,1\}^n$.
	$f$ is $(1,{\rm id}, {\rm id})$-network representable if and only if $\overline{f}$ is $(1,{\rm id}, {\rm id})$-network representable.
\end{lem}
\begin{proof}
	Suppose that $f$ is represented by a network $G = (\{s,t\} \cup V, A; c)$.
	Then $\overline{f}$ is represented by $\overline{G} = (\{s,t\} \cup V, \overline{A}; \overline{c})$, where
	\begin{align*}
		&\overline{A} := \{(j, i) \mid i, j \in V,\ (i, j) \in A\} \cup \{(i, t) \mid (s, i) \in A\} \cup \{(s, i) \mid (i, t) \in A\},\\
		&\overline{c}(i, j) := \begin{cases}
			c(j, i) & \text{if $i,j \in V$},\\
			c(s, i) & \text{if $j = t$}, \\
			c(j, t) & \text{if $i = s$}.
		\end{cases}
	\end{align*}
\end{proof}

\begin{prop}\label{prop:mono}
$\overline{\average{\Gamma_{\rm sub,2}}}_{\{0,1\}} \subseteq \average{\Gamma_{\rm sub,2}} \cup \Gamma_{\rm mono}$.
\end{prop}
\begin{proof}
Take arbitrary positive integer $k$.
There are three cases of a map $\sigma : \{0,1\} \rightarrow \{0,1\}^k$:
(i) $\sigma(0) \wedge \sigma(1) = \sigma(0)$,
(ii) $\sigma(0) \wedge \sigma(1) = \sigma(1)$,
and (iii) $\sigma(0) \neq \left(\sigma(0) \wedge \sigma(1)\right) \neq \sigma(1)$.
We prove $\average{\Gamma_{{\rm sub},2}}^k_{(\rho, \sigma)} \subseteq \average{\Gamma_{{\rm sub},2}} \cup \Gamma_{\rm mono}$ for all cases of $\sigma$ and a retraction $\rho : \{0,1\}^k \rightarrow \{\sigma(0), \sigma(1)\}$ in the following.
Suppose that the arity of a function $f$ is equal to $n$.

\paragraph{(i) $\sigma(0) \wedge \sigma(1) = \sigma(0)$.}
Let us prove $f \in \average{\Gamma_{{\rm sub},2}}$ for $f \in \average{\Gamma_{{\rm sub},2}}^k_{(\rho, \sigma)}$.
Let $h : \{0,1\}^{k+1} \rightarrow \overline{\mathbf{Q}}$ be a function defined by
\begin{align*}
h(v, x) := \begin{cases} 0 & \text{if $(v, x) \in \{\left(\sigma(0), 0\right), \left(\sigma(1), 1\right)\}$}, \\ +\infty & \text{otherwise}. \end{cases}
\end{align*}
This $h$ is $(1,\textrm{id},\textrm{id})$-network representable by the network $G = (\{s,t\} \cup \{1,2,\dots, k\} \cup \{x\}, A_0 \cup A_1 \cup A'; c)$ defined by $c(e) := + \infty$ for all $e \in A_0 \cup A_1 \cup A'$, where
\begin{align*}
&A_0 := \{(i,t) \mid (\sigma(0))_i = (\sigma(1))_i = 0 \text{ for } i \in [k]\},\\
&A_1 := \{(s,i) \mid (\sigma(0))_i = (\sigma(1))_i = 1 \text{ for } i \in [k]\},\\
&A' := \{(i, x), (x, i) \mid (\sigma(0))_i = 0 \text{ and } (\sigma(1))_i = 1 \text{ for } i \in [k]\}.
\end{align*}
Then $h \in \average{\Gamma_{\rm sub,2}}$.
Since $f \in \average{\Gamma_{{\rm sub},2}}^k_{(\rho, \sigma)}$,
there exists $g \in \average{\Gamma_{\rm sub,2}}^k \subseteq \average{\Gamma_\textrm{sub,2}}$ satisfying $g(\rho(v)) \leq g(v)$ for $v \in \textrm{dom }g$ and
\begin{align*}
f(x_1,x_2,\dots, x_n) = g(\sigma(x_1), \sigma(x_2),\dots,\sigma(x_n)) \qquad (x_1,x_2,\dots, x_n \in \{0,1\}).
\end{align*}
By using $h$, we obtain
\begin{align*}
g(\sigma(x_1),\sigma(x_2), \dots,\sigma(x_n)) = \min_{v_1,v_2, \dots, v_n \in \{0, 1\}^k} \left( g(v_1,v_2,\dots, v_n) + h(v_1, x_1) + \cdots + h(v_n, x_n) \right)
\end{align*}
for all $x_1,x_2,\dots,x_n \in D$.
Therefore it holds that
\begin{align*}
f(x_1,x_2,\dots, x_n) = \min_{v_1,v_2, \dots, v_n \in \{0, 1\}^k} \left( g(v_1,v_2,\dots, v_n) + h(v_1, x_1) + \cdots + h(v_n, x_n) \right).
\end{align*}
By $g,h \in \average{\Gamma_{{\rm sub},2}}$, we obtain $f \in \average{\Gamma_{{\rm sub},2}}$.

\paragraph{(ii) $\sigma(0) \wedge \sigma(1) = \sigma(1)$.}
Let us prove $f \in \average{\Gamma_{{\rm sub},2}}$ for $f \in \average{\Gamma_{{\rm sub},2}}^k_{(\rho, \sigma)}$.
By Lemma~\ref{lem:network rep},
it suffices to show $\overline{f} \in \average{\Gamma_{{\rm sub},2}}$.
By Lemma~\ref{lem:f and overline f},
we obtain $\overline{f} \in \average{\Gamma_{{\rm sub},2}}^k_{(\rho, \overline{\sigma})}$.
Moreover, since $\sigma(0) \wedge \sigma(1) = \sigma(1)$,
we have $\overline{\sigma}(0) \wedge \overline{\sigma}(1) = \overline{\sigma}(0)$.
Thus we obtain $\overline{f} \in \average{\Gamma_{{\rm sub},2}}$ by the case of (i).

\paragraph{(iii) $\sigma(0) \neq \left(\sigma(0) \wedge \sigma(1)\right) \neq \sigma(1)$.}
In this case, there are four cases as follows:
\begin{description}
\item[(iii-1)] $\rho\left(\sigma(0) \land \sigma(1)\right) = \sigma(0)$ and $\rho\left(\sigma(0) \lor \sigma(1)\right) = \sigma(1)$,
\item[(iii-2)] $\rho\left(\sigma(0) \land \sigma(1)\right) = \sigma(1)$ and $\rho\left(\sigma(0) \lor \sigma(1)\right) = \sigma(0)$,
\item[(iii-3)] $\rho\left(\sigma(0) \land \sigma(1)\right) = \rho\left(\sigma(0) \lor \sigma(1)\right) = \sigma(0)$,
\item[(iii-4)] $\rho\left(\sigma(0) \land \sigma(1)\right) = \rho\left(\sigma(0) \lor \sigma(1)\right) = \sigma(1)$.
\end{description}

\paragraph{(iii-1) $\rho\left(\sigma(0) \land \sigma(1)\right) = \sigma(0)$ and $\rho\left(\sigma(0) \lor \sigma(1)\right) = \sigma(1)$.}
Let us prove $f \in \average{\Gamma_{{\rm sub},2}}$ for $f \in \average{\Gamma_{{\rm sub},2}}^k_{(\rho, \sigma)}$.
Since $f \in \average{\Gamma_{{\rm sub},2}}^k_{(\rho, \sigma)}$,
there exists $g \in \average{\Gamma_{\rm sub,2}}^k \subseteq \average{\Gamma_\textrm{sub,2}}$ satisfying $g(\rho(v)) \leq g(v)$ for $v \in \textrm{dom }g$ and
\begin{align}\label{eq:f=g}
f(x_1,x_2,\dots, x_n) = g(\sigma(x_1), \sigma(x_2),\dots,\sigma(x_n)) \qquad (x_1,x_2,\dots, x_n \in \{0,1\}).
\end{align}
Let $\sigma' : \{0,1\} \rightarrow \{0,1\}^k$ be a function defined by
\begin{align*}
\sigma'(x) := \begin{cases} \sigma(0) \vee \sigma(1) & \text{if $x = 1$}, \\ \sigma(0) \wedge \sigma(1) & \text{if $x = 0$}. \end{cases}
\end{align*}
Here the following claim holds.
\begin{cl}\label{cl:sigma}
$f(x_1,x_2,\dots, x_n) = g(\sigma'(x_1),\sigma'(x_2), \dots, \sigma'(x_n)).$
\end{cl}
\begin{proof}[Proof of Claim~\ref{cl:sigma}]
Take arbitrary $x_1,x_2,\dots, x_n \in \{0,1\}$.
Then we obtain
\begin{align}
&f(x_1,x_2, \dots, x_n) + f(1,1, \dots, 1)\label{cl:sigma0}\\
&= g(\sigma(x_1),\sigma(x_2), \dots, \sigma(x_n)) + g(\sigma(1),\sigma(1), \dots, \sigma(1))\label{cl:sigma1}\\
&\geq g(\sigma(x_1) \wedge \sigma(1), \dots, \sigma(x_n) \wedge \sigma(1)) + g(\sigma(x_1) \vee \sigma(1), \dots, \sigma(x_n) \vee \sigma(1))\label{cl:sigma2}\\
&\geq g(\sigma(x_1),\sigma(x_2), \dots, \sigma(x_n)) + g(\sigma(1),\sigma(1), \dots, \sigma(1))\label{cl:sigma3}\\
&= f(x_1,x_2, \dots, x_n) + f(1,1, \dots, 1).\label{cl:sigma4}
\end{align}
Indeed, (\ref{cl:sigma0}) = (\ref{cl:sigma1}) is obvious by (\ref{eq:f=g}), and (\ref{cl:sigma1}) $\geq$ (\ref{cl:sigma2}) follows from the submodularity of $g$.
By the assumption of $\rho$, it holds that
$\rho(\sigma(x) \wedge \sigma(1)) = \sigma(x)$ and $\rho(\sigma(x) \vee \sigma(1)) = \sigma(1)$.
Hence we have (\ref{cl:sigma2}) $\geq$ (\ref{cl:sigma3}).
(\ref{cl:sigma3}) = (\ref{cl:sigma4}) is also obvious by (\ref{eq:f=g}).
This means that all inequalities are equalities.
Then it holds that
\begin{align}\label{eq:f(x)}
f(x_1,x_2, \dots, x_n) = g(\sigma(x_1) \wedge \sigma(1),\sigma(x_2) \wedge \sigma(1), \dots, \sigma(x_n) \wedge \sigma(1)).
\end{align}
Also we obtain
\begin{align}
&f(0,0, \dots, 0) + f(\overline{x_1},\overline{x_2}, \dots, \overline{x_n})\label{cl:sigma'0}\\
&= g(\sigma(0),\sigma(0), \dots, \sigma(0)) + g(\sigma(\overline{x_1}),\sigma(\overline{x_2}), \dots, \sigma(\overline{x_n}))\label{cl:sigma'1}\\
&\geq g(\sigma(\overline{x_1}) \wedge \sigma(0), \dots, \sigma(\overline{x_n}) \wedge \sigma(0)) + g(\sigma(\overline{x_1}) \vee \sigma(0), \dots, \sigma(\overline{x_n}) \vee \sigma(0))\label{cl:sigma'2}\\
&\geq g(\sigma(0),\sigma(0), \dots, \sigma(0)) + g(\sigma(\overline{x_1}),\sigma(\overline{x_2}), \dots, \sigma(\overline{x_n}))\label{cl:sigma'3}\\
&= f(0,0, \dots, 0) + f(\overline{x_1},\overline{x_2}, \dots, \overline{x_n}).\label{cl:sigma'4}
\end{align}
Indeed, (\ref{cl:sigma'0}) = (\ref{cl:sigma'1}) is obvious by (\ref{eq:f=g}), and (\ref{cl:sigma'1}) $\geq$ (\ref{cl:sigma'2}) follows from the submodularity of $g$.
By the assumption of $\rho$, it holds that
$\rho(\sigma(x) \wedge \sigma(0)) = \sigma(0)$ and $\rho(\sigma(x) \vee \sigma(0)) = \sigma(x)$.
Hence we have (\ref{cl:sigma'2}) $\geq$ (\ref{cl:sigma'3}).
(\ref{cl:sigma'3}) = (\ref{cl:sigma'4}) is also obvious by (\ref{eq:f=g}).
This means that all inequalities are equalities.
Then it holds that
\begin{align}\label{eq:f(0)}
f(0,0, \dots, 0) = g(\sigma(\overline{x_1}) \wedge \sigma(0),\sigma(\overline{x_2}) \wedge \sigma(0), \dots, \sigma(\overline{x_n}) \wedge \sigma(0)).
\end{align}
Hence we have
\begin{align}
&f(0,0, \dots, 0) + f(x_1,x_2, \dots, x_n)\label{cl:sigma''0}\\
&= g(\sigma(\overline{x_1}) \wedge \sigma(0), \dots, \sigma(\overline{x_n}) \wedge \sigma(0)) + g(\sigma(x_1) \wedge \sigma(1), \dots, \sigma(x_n) \wedge \sigma(1))\label{cl:sigma''1}\\
&\geq g(\sigma(0) \wedge \sigma(1), \dots, \sigma(0) \wedge \sigma(1)) + g(\sigma'(x_1), \dots, \sigma'(x_n))\label{cl:sigma''2}\\
&\geq g(\sigma(0),\sigma(0), \dots, \sigma(0)) + g(\sigma(x_1),\sigma(x_2), \dots, \sigma(x_n))\label{cl:sigma''3}\\
&= f(0,0, \dots, 0) + f(x_1,x_2, \dots, x_n).\label{cl:sigma''4}
\end{align}
By (\ref{eq:f(x)}) and (\ref{eq:f(0)}), it holds that (\ref{cl:sigma''0}) = (\ref{cl:sigma''1}).
(\ref{cl:sigma''1}) $\geq$ (\ref{cl:sigma''2}) follows from the submodularity of $g$.
(\ref{cl:sigma''2}) $\geq$ (\ref{cl:sigma''3}) follows from the assumption of $\rho$.
(\ref{cl:sigma''3}) = (\ref{cl:sigma''4}) is obvious by (\ref{eq:f=g}).
This means that all inequalities are equalities.
Hence we obtain
$f(x_1,x_2, \dots, x_n) = g(\sigma'(x_1),\sigma'(x_2), \dots, \sigma'(x_n))$.
\end{proof}

We define $\rho' : \{0,1\}^k \rightarrow \{0,1\}^k$ by
\begin{align*}
\rho'(x) := \begin{cases} \sigma'(1) & \text{if $\rho(x) = \sigma(1)$},\\ \sigma'(0) & \text{if $\rho(x) = \sigma(0)$}. \end{cases}
\end{align*}
By Claim~\ref{cl:sigma}, a function $f \in \average{\Gamma_{{\rm sub},2}}^k_{(\rho, \sigma)}$ is also representable by $(\Gamma_{{\rm sub},2},\rho', \sigma')$.
Hence $f \in \average{\Gamma_{{\rm sub},2}}^k_{(\rho', \sigma')}$.
Here it holds that $\sigma'(0) \wedge \sigma'(1) = \sigma'(1)$.
This means that $\average{\Gamma_{{\rm sub},2}}^k_{(\rho', \sigma')}$ is in the case (i).
Therefore we obtain $f \in \average{\Gamma_{{\rm sub},2}}$.

\paragraph{(iii-2) $\rho\left(\sigma(0) \land \sigma(1)\right) = \sigma(1)$ and $\rho\left(\sigma(0) \lor \sigma(1)\right) = \sigma(0)$.}
We also prove $f \in \average{\Gamma_{{\rm sub},2}}$ for $f \in \average{\Gamma_{{\rm sub},2}}^k_{(\rho, \sigma)}$.
By Lemma~\ref{lem:network rep},
it suffices to show $\overline{f} \in \average{\Gamma_{{\rm sub},2}}$.
By Lemma~\ref{lem:f and overline f},
we obtain $\overline{f} \in \average{\Gamma_{{\rm sub},2}}^k_{(\rho, \overline{\sigma})}$.
Moreover, since $\rho\left(\sigma(0) \land \sigma(1)\right) = \sigma(1)$ and $\rho\left(\sigma(0) \lor \sigma(1)\right) = \sigma(0)$,
it holds that $\rho\left(\overline{\sigma}(0) \land \overline{\sigma}(1)\right) = \overline{\sigma}(0)$ and $\rho\left(\overline{\sigma}(0) \lor \overline{\sigma}(1)\right) = \overline{\sigma}(1)$.
Thus we obtain $\overline{f} \in \average{\Gamma_{{\rm sub},2}}$ by the case of (iii-1).

\paragraph{(iii-3) $\rho\left(\sigma(0) \land \sigma(1)\right) = \rho\left(\sigma(0) \lor \sigma(1)\right) = \sigma(0)$.}
We prove that for all $f \in \average{\Gamma_{{\rm sub},2}}^k_{(\rho, \sigma)}$, $f$ is a monotone non-decreasing function.
For every $i \in [n]$ and $x_1,\dots, x_{i-1}, x_{i+1}, \dots, x_n \in \{0,1\}$, it holds that
\begin{align}
&f(x_1,\dots, x_{i-1}, 1, x_{i+1}, \dots, x_n) + f(x_1,\dots, x_{i-1}, 0, x_{i+1}, \dots, x_n)\label{iii-3:1}\\
&= g(\sigma_1,\dots,\sigma_{i-1}, \sigma(1), \sigma_{i+1}, \dots, \sigma_n) + g(\sigma_1,\dots,\sigma_{i-1}, \sigma(0), \sigma_{i+1}, \dots, \sigma_n)\label{iii-3:2}\\
&\geq g(\sigma_1,\dots,\sigma_{i-1}, \sigma(1) \wedge \sigma(0), \sigma_{i+1}, \dots, \sigma_n) + g(\sigma_1,\dots,\sigma_{i-1}, \sigma(1) \vee \sigma(0), \sigma_{i+1}, \dots, \sigma_n)\label{iii-3:3}\\
&\geq 2 g(\sigma_1,\dots,\sigma_{i-1}, \sigma(0), \sigma_{i+1}, \dots, \sigma_n)\label{iii-3:4}\\
&= 2 f(x_1,\dots,x_{i-1}, 0, x_{i+1}, \dots, x_n).\label{iii-3:5}
\end{align}
Here $\sigma_j := \sigma(x_j)$ for $j \in [n] \setminus i$.
Indeed, (\ref{iii-3:2}) $\geq$ (\ref{iii-3:3}) follows from the submodularity of $g$,
and (\ref{iii-3:3}) $\geq$ (\ref{iii-3:4}) follows from the assumption of $\rho$.
Therefore for $i \in [n]$ and $x_1,\dots, x_{i-1}, x_{i+1}, \dots, x_n \in \{0,1\}$, we have
\begin{align*}
f(x_1,\dots,x_{i-1}, 1, x_{i+1}, \dots, x_n) \geq f(x_1,\dots,x_{i-1}, 0, x_{i+1}, \dots, x_n).
\end{align*}
This means $f$ is a monotone non-decreasing function.

\paragraph{(iii-4) $\rho\left(\sigma(0) \land \sigma(1)\right) = \rho\left(\sigma(0) \lor \sigma(1)\right) = \sigma(1)$.}
We prove that for all $f \in \average{\Gamma_{{\rm sub},2}}^k_{(\rho, \sigma)}$, $f$ is a monotone non-increasing function.
By Lemma~\ref{lem:f and overline f},
we obtain $\overline{f} \in \average{\Gamma_{{\rm sub},2}}^k_{(\rho, \overline{\sigma})}$.
Moreover, since $\rho\left(\sigma(0) \land \sigma(1)\right) = \rho\left(\sigma(0) \lor \sigma(1)\right) = \sigma(1)$,
it holds that $\rho\left(\overline{\sigma}(0) \land \overline{\sigma}(1)\right) = \rho\left(\overline{\sigma}(0) \lor \overline{\sigma}(1)\right) = \overline{\sigma}(0)$.
Thus $\overline{f}$ is a monotone non-decreasing function by the case of (iii-3).
Hence $f$ is a monotone non-increasing function.
\end{proof}

\begin{prop}\label{prop:second subsetneq}
$\overline{\average{\Gamma_{\rm sub,2}}}_{\{0,1\}} \subsetneq \average{\Gamma_{\rm sub,2}} \cup \Gamma_{\rm mono}$.
\end{prop}
The proof of Proposition~\ref{prop:second subsetneq} is given in Section~\ref{subsec:prop: second subsetneq}.

By Lemma~\ref{lem:first subsetneq}, Proposition~\ref{prop:mono}, and Proposition~\ref{prop:second subsetneq}, we obtain Theorem~\ref{thm:{0,1}main}.

\subsection{Proof of Theorem~\ref{thm:{0,1}eqExpressive}}\label{subsec:thm:{0,1}eqExpressive}
By the proof of Proposition~\ref{prop:mono}, it holds that $ \average{\Gamma_{{\rm sub},2}}^k_{(\rho, \sigma)} \subseteq \average{\Gamma_{{\rm sub},2}}$ in the cases of (i), (ii), (iii-1), and (iii-2) in Section~\ref{subsec:thm:{0,1}mainExpressive}.
Hence we consider only the two cases as follows:
\begin{description}
\item[(iii-3)] $\sigma(0) \neq \left(\sigma(0) \wedge \sigma(1)\right) \neq \sigma(1)$ and $\rho\left(\sigma(0) \land \sigma(1)\right) = \rho\left(\sigma(0) \lor \sigma(1)\right) = \sigma(0)$,
\item[(iii-4)] $\sigma(0) \neq \left(\sigma(0) \wedge \sigma(1)\right) \neq \sigma(1)$ and $\rho\left(\sigma(0) \land \sigma(1)\right) = \rho\left(\sigma(0) \lor \sigma(1)\right) = \sigma(1)$.
\end{description}

\paragraph{The case (iii-3).}
We prove $\average{\Gamma_{{\rm sub},2}}^k_{(\rho, \sigma)} \subseteq \average{\Gamma_{{\rm sub},2}}^2_{(\rho^*, \sigma^*_1)}$.
Let $S_0, S_1, A, B$ be index sets defined by
\begin{align*}
&S_0 := \{ i \mid (\sigma(0))_i = (\sigma(1))_i = 0 \text{ for } i \in [k] \},\\
&S_1 := \{ i \mid (\sigma(0))_i = (\sigma(1))_i = 1 \text{ for } i \in [k] \},\\
&A := \{ j \mid (\sigma(0))_j = 0 \text{ and } (\sigma(1))_j = 1 \text{ for } i \in [k] \},\\
&B := \{ k \mid (\sigma(0))_j = 1 \text{ and } (\sigma(1))_j = 0 \text{ for } i \in [k] \}.
\end{align*}
Let $j_0$ be the minimum index in $A$, and $k_0$ the minimum index in $B$.
Let $S := (S_0 \cup S_1 \cup A \cup B) \setminus \{j_0, k_0\}$.
Furthermore we define functions $h_0 : \{0, 1\} \rightarrow \overline{\mathbf{Q}},\ h_1 : \{0, 1\} \rightarrow \overline{\mathbf{Q}},\ h_2 : \{0, 1\}^2 \rightarrow \overline{\mathbf{Q}}$ by
\begin{align*}
h_0(x) := \begin{cases} 0 & \text{if $x = 0$}, \\ +\infty & \text{if $x = 1$}, \end{cases}\quad
h_1(x) := \begin{cases} 0 & \text{if $x = 1$}, \\ +\infty & \text{if $x = 0$}, \end{cases}\quad
h_2(x, y) := \begin{cases} 0 & \text{if $x = y$}, \\ +\infty & \text{otherwise}. \end{cases}
\end{align*}
$h_0$, $h_1$, and $h_2$ are $(1,\textrm{id},\textrm{id})$-network representable.
Indeed, $h_0$ is represented by the network $\left(\{s,t,1\}, \{(1,t)\}\right)$, where the edge capacity of $(1,t)$ is equal to $+\infty$,
$h_1$ is represented by the network $\left(\{s,t,1\}, \{(s,1)\}\right)$, where the edge capacity of $(s,1)$ is equal to $+\infty$,
and $h_2$ is represented by the network $\left(\{s,t,1,2\}, \{(1,2), (2,1)\}\right)$, where the edge capacities of $(1,2)$ and $(2,1)$ are equal to $+\infty$.
Take arbitrary $f \in \average{\Gamma_{{\rm sub},2}}^k_{(\rho, \sigma)}$.
Then there exists $g \in \average{\Gamma_{{\rm sub},2}}^k \subseteq \average{\Gamma_\textrm{sub,2}}$ satisfying $g(\rho(v)) \leq g(v)$ for $v \in \textrm{dom }g$ and
\begin{align*}
f(x_1,x_2,\dots, x_n) = g(\sigma(x_1), \sigma(x_2),\dots,\sigma(x_n)) \qquad (x_1,x_2,\dots, x_n \in \{0,1\}).
\end{align*}
Let $g' : \{0, 1\}^{kn} \rightarrow \overline{\mathbf{Q}}$ be a function defined by
\begin{align*}
g'(v_1,v_2, \dots, v_n) :=\ &g(v_1,v_2,\dots,v_n) + \sum_{i}\sum_{i_0 \in S_0}h_0((v_i)_{i_0}) + \sum_{i}\sum_{i_1 \in S_1}h_1((v_i)_{i_1})\\
&+ \sum_{i}\sum_{j \in A}h_2((v_i)_{j}, (v_i)_{j_0}) + \sum_{i}\sum_{k \in B}h_2((v_i)_{k}, (v_i)_{k_0})
\end{align*}
for $v_1,v_2,\dots,v_n \in \{0,1\}^k$.
By the definition of $g'$, we have
\begin{align*}
g'(v) =
\begin{cases}
g(v) & \text{if $v_i \in \{\sigma(0), \sigma(1), \sigma(0) \wedge \sigma(1), \sigma(0) \vee \sigma(1)\}$ for each $i \in [n]$},\\
+\infty & \text{otherwise},
\end{cases}
\end{align*}
for all $v = (v_1,v_2,\dots, v_n) \in \{0,1\}^{kn}$.
We notice that $g'$ only depends on $2n$ elements $\{(v_i)_{j_0},(v_i)_{k_0}\}_{i \in [n]}$.
Hence let $g'' : \{0, 1\}^{2n} \rightarrow \overline{\mathbf{Q}}$ be a function defined by
\begin{align*}
g''(u_1,u_2, \dots, u_n) := &\min_{\substack{(v_1)_{j_0} = (u_1)_1, \dots, (v_n)_{j_0} = (u_n)_1\\ (v_1)_{k_0} = (u_1)_2, \dots, (v_n)_{k_0} = (u_n)_2}} g'(v_1,v_2,\dots, v_n)
\end{align*}
for $u_1,u_2,\dots, u_n \in \{0, 1\}^2$.
Since $g, h_0, h_1, h_2 \in \average{\Gamma_{{\rm sub},2}}$, we have $g', g'' \in \average{\Gamma_{{\rm sub},2}}$.
Furthermore we have $g''(u_1,u_2,\dots, u_n) \geq g''(\rho^*(u_1),\rho^*(u_2), \dots, \rho^*(u_n))$ for $u_1,u_2, \dots, u_n \in \{0,1\}^2$.
Indeed, it holds that
\begin{align*}
g''(u_1,u_2,\dots, u_n) = g(v_1,v_2, \dots, v_n) \qquad (u_1,u_2, \dots, u_n \in \{0,1\}^2)
\end{align*}
by the definition of $g''$,
where
\begin{align*}
v_i := \begin{cases}
\sigma(0) & \text{if $u_i = (0,1)$},\\
\sigma(1) & \text{if $u_i = (1,0)$},\\
\sigma(0) \wedge \sigma(1) & \text{if $u_i = (0,0)$},\\
\sigma(0) \vee \sigma(1) & \text{if $u_i = (1,1)$},
\end{cases} \qquad (i \in [n]).
\end{align*}
Hence it holds that
\begin{align*}
	g''(u_1,u_2,\dots, u_n) &= g(v_1,v_2, \dots, v_n)\\
	&\geq g(\rho(v_1),\rho(v_2), \dots, \rho(v_n))\\
	&= g''(\rho^*(u_1),\rho^*(u_2), \dots, \rho^*(u_n))
\end{align*}
by the assumption of $\rho$.
By using $g''$, $f$ is represented by
\begin{align*}
f(x_1,x_2,\dots, x_n) = g''(\sigma^*_1(x_1),\sigma^*_1(x_2), \dots, \sigma^*_1(x_n)) \qquad (x_1,x_2,\dots,x_n \in D).
\end{align*}
This means that $f \in \average{\Gamma_{{\rm sub},2}}^2_{(\rho^*, \sigma^*_1)}$.

\paragraph{The case (iii-4).}
We prove $\average{\Gamma_{{\rm sub},2}}^k_{(\rho, \sigma)} \subseteq \average{\Gamma_{{\rm sub},2}}^2_{(\rho^*, \sigma^*_2)}$.
Take any $f \in \average{\Gamma_{{\rm sub},2}}^k_{(\rho, \sigma)}$.
By Lemma~\ref{lem:f and overline f},
we obtain $\overline{f} \in \average{\Gamma_{{\rm sub},2}}^k_{(\rho, \overline{\sigma})}$.
Moreover, since $\sigma(0) \neq \left(\sigma(0) \wedge \sigma(1)\right) \neq \sigma(1)$ and $\rho\left(\sigma(0) \land \sigma(1)\right) = \rho\left(\sigma(0) \lor \sigma(1)\right) = \sigma(1)$,
it holds that $\overline{\sigma}(0) \neq \left(\overline{\sigma}(0) \wedge \overline{\sigma}(1)\right) \neq \overline{\sigma}(1)$ and $\rho\left(\overline{\sigma}(0) \land \overline{\sigma}(1)\right) = \rho\left(\overline{\sigma}(0) \lor \overline{\sigma}(1)\right) = \overline{\sigma}(0)$.
Thus $\overline{f} \in \average{\Gamma_{{\rm sub},2}}^2_{(\rho^*, \sigma^*_1)}$ holds by the case of (iii-3).
Hence we obtain $f \in \average{\Gamma_{{\rm sub},2}}^2_{(\rho^*, \overline{\sigma^*_1})} = \average{\Gamma_{{\rm sub},2}}^2_{(\rho^*, \sigma^*_2)}$.

Then it holds that $\overline{\average{\Gamma_{{\rm sub},2}}}_{\{0,1\}} = \average{\Gamma_{{\rm sub},2}} \cup \average{\Gamma_{{\rm sub},2}}^2_{(\rho^*, \sigma^*_1)} \cup \average{\Gamma_{{\rm sub},2}}^2_{(\rho^*, \sigma^*_2)}$.

\subsection{Proof of Proposition~\ref{prop:second subsetneq}}\label{subsec:prop: second subsetneq}
Let $f : \{0,1\}^3 \rightarrow \mathbf{Q}$ be a monotone non-decreasing function defined by
$f(1,1,1) := 1$ and $f(x) := 0$ for other $x$.
The function $f$ is not submodular.
Therefore $f \not\in \average{\Gamma_{\rm sub,2}}$.
Hence it suffices to prove $f \not\in \average{\Gamma_{{\rm sub},2}}^2_{(\rho^*, \sigma^*_1)}$ by the proof of Theorem~\ref{thm:{0,1}eqExpressive} (the case (iii-3)).
Suppose to the contrary that $f \in \average{\Gamma_{{\rm sub},2}}^2_{(\rho^*, \sigma^*_1)}$.
Then there exists $g \in \average{\Gamma_{{\rm sub},2}}$ satisfying $g(\sigma^*_1(1), \sigma^*_1(1), \sigma^*_1(1)) = g(1,0,1,0,1,0) = 1$,
$g(x) = 0$ for $x \in A$,
and $\min_x g(x) = 0$,
where a set $A \subseteq \{0,1\}^6$ is defined by
\begin{align*}
A := \{&\left(\sigma^*_1(0), \sigma^*_1(0), \sigma^*_1(0)\right), \left(\sigma^*_1(0), \sigma^*_1(0), \sigma^*_1(1)\right), \left(\sigma^*_1(0), \sigma^*_1(1), \sigma^*_1(0)\right), \left(\sigma^*_1(0), \sigma^*_1(1), \sigma^*_1(1)\right),\\
&\left(\sigma^*_1(1), \sigma^*_1(0), \sigma^*_1(0)\right), \left(\sigma^*_1(1), \sigma^*_1(0), \sigma^*_1(1)\right), \left(\sigma^*_1(1), \sigma^*_1(1), \sigma^*_1(0)\right)\}\\
= \{ &(0,1,0,1,0,1), (0,1,0,1,1,0), (0,1,1,0,0,1),(0,1,1,0,1,0),\\
&(1,0,0,1,0,1), (1,0,0,1,1,0), (1,0,1,0,0,1)\}.
\end{align*}
Let $C_{\wedge, \vee}(A)$ denote the minimum subset $X$ of $\{0,1\}^6$ containing $A$ such that $x \wedge y, x \vee y \in X$ for all $x,y \in X$.
By the submodularity of $g$, it holds that $g(x) = 0$ for $x \in C_{\wedge, \vee}(A)$.
Therefore it should hold that $(1,0,1,0,1,0) \not\in C_{\wedge, \vee}(A)$.
However by
\begin{align*}
&(1,0,0,1,1,0) \wedge (1,0,1,0,0,1) = (1,0,0,0,0,0),\\
&(0,1,1,0,1,0) \wedge (1,0,1,0,0,1) = (0,0,1,0,0,0),\\
&(0,1,1,0,1,0) \wedge (1,0,0,1,1,0) = (0,0,0,0,1,0),\\
&(1,0,0,0,0,0) \vee (0,0,1,0,0,0) \vee (0,0,0,0,1,0) = (1,0,1,0,1,0),
\end{align*}
we have $(1,0,1,0,1,0) \in C_{\wedge, \vee}(A)$.
This is a contradiction to the existence of such a function $g$.
Then it holds that $f \not\in \average{\Gamma_{{\rm sub},2}}^2_{(\rho^*, \sigma^*_1)}$.

\subsection{Proof of Theorem~\ref{thm:3ary2subExpressive}}\label{subsec:thm:3ary2subExpressive}
Let $f : \{0,-1,1\}^3 \rightarrow \mathbf{Q}$ be a bisubmodular function defined by
$f(0,0,0) := -1$, $f(0,1,1) = f(1,0,1) = f(1,1,0) := 1$, $f(1,1,1) := 2$, and $f(x) = 0$ for other $x$.
It suffices to prove $f \not\in \average{\Gamma_{\rm sub}}^2_{(\rho_2, \sigma_2)}$.

$\omega_2 : {\rm Pol}^{(4)}(\average{\Gamma_{\rm sub}}^2_{(\rho_2, \sigma_2)}) \rightarrow \mathbf{Q}$ defined as the following is a weighted polymorphism of $\average{\Gamma_{\rm sub}}^2_{(\rho_2, \sigma_2)}$:
\begin{align*}
\omega_2(\varphi) := \begin{cases}
-1 & \text{if $\varphi \in \{e_1^{(4)}, e_2^{(4)}, e_3^{(4)}, e_4^{(4)}\}$},\\
1 & \text{if $\varphi \in \{\varphi_1, \varphi_2, \varphi_3, \varphi_4\}$},\\
0 & \text{otherwise}.
\end{cases}
\end{align*}
Here $\varphi_1, \varphi_2, \varphi_3, \varphi_4 : \{0, -1, 1\}^4 \rightarrow \{0,-1,1\}^4$ is defined by
\begin{align*}
&\varphi_1(a,b,c,d) := \sigma_2(\rho_2(\sigma_2^{-1}(a) \wedge \sigma_2^{-1}(b))),\\
&\varphi_2(a,b,c,d) := \sigma_2(\rho_2((\sigma_2^{-1}(a) \vee \sigma_2^{-1}(b)) \wedge \sigma_2^{-1}(c))),\\
&\varphi_3(a,b,c,d) := \sigma_2(\rho_2((\sigma_2^{-1}(a) \vee \sigma_2^{-1}(b) \vee \sigma_2^{-1}(c)) \wedge \sigma_2^{-1}(d))),\\
&\varphi_4(a,b,c,d) := \sigma_2(\rho_2(\sigma_2^{-1}(a) \vee \sigma_2^{-1}(b) \vee \sigma_2^{-1}(c) \vee \sigma_2^{-1}(d))).
\end{align*}
Indeed, a function $g \in \average{\Gamma_{\rm sub}}^k$ such that $g(v) \geq g(\rho_2(v))$ for $v \in {\rm dom}\ g$ satisfies the following inequalities for all $x, y, z, w \in \{0,1\}^{2n}$;
\begin{align}
&g(x) + g(y) + g(z) + g(w)\notag\\
&\geq g(x \wedge y) + g(x \vee y) + g(z) + g(w)\label{eq:bi1}\\
&\geq g(x \wedge y) + g((x \vee y) \wedge z) + g(x \vee y \vee z) + g(w)\label{eq:bi2}\\
&\geq g(x \wedge y) + g((x \vee y) \wedge z) + g((x \vee y \vee z) \wedge w) + g(x \vee y \vee z \vee w)\label{eq:bi3}\\
&\geq g(x') + g(y') + g(z') + g(w')\label{eq:bi4}
\end{align}
Here let $x' := \rho_2(x \wedge y),\ y' := \rho_2((x \vee y) \wedge z),\ z' := \rho_2((x \vee y \vee z) \wedge w),\ w' := \rho_2(x \vee y \vee z \vee w)$.
(\ref{eq:bi1})--(\ref{eq:bi3}) follow from the submodularity, and (\ref{eq:bi4}) follows from the definition of $g$.
This means that $h(a) + h(b) + h(c) + h(d) \geq h(\varphi_1(a,b,c,d)) + h(\varphi_2(a,b,c,d)) + h(\varphi_3(a,b,c,d)) + h(\varphi_4(a,b,c,d))$ holds for any $h \in \average{\Gamma_{\rm sub}}^2_{(\rho_2, \sigma_2)}$ and $a,b,c,d \in \textrm{dom }h$.

Let $x := (0,0,1,0,0,1),\ y := (0,0,0,1,0,1),\ z := (0,0,0,1,1,0),\ w := (0,1,1,0,1,0)$, that is,
$\sigma_2^{-1}(x) = (0,1,-1),\ \sigma_2^{-1}(y) = (0,-1,-1),\ \sigma_2^{-1}(z) = (0,-1,1),\ \sigma_2^{-1}(w) = (-1,1,1)$.
In this case, it holds that $\rho_2(x \wedge y) = (0,0,0,0,0,1),\ \rho_2((x \vee y) \wedge z) = (0,0,0,1,0,0),\ \rho_2((x \vee y \vee z) \wedge w) = (0,0,1,0,1,0),\ \rho_2(x \vee y \vee z \vee w) = (0,1,0,0,0,0)$,
that is, $\sigma_2^{-1}(\rho_2(x \wedge y)) = (0,0,-1),\ \sigma_2^{-1}(\rho_2((x \vee y) \wedge z)) = (0,-1,0),\ \sigma_2^{-1}(\rho_2((x \vee y \vee z) \wedge w)) = (0,1,1),\ \sigma_2^{-1}(\rho_2(x \vee y \vee z \vee w)) = (-1,0,0)$.
Hence we have
\begin{align*}
0 =\ &f(0,1,-1)+f(0,-1,-1)+f(0,-1,1)+f(-1,1,1)\\
<\ &f(0,0,-1)+f(0,-1,0)+f(0,1,1)+f(-1,0,0) = 1.
\end{align*}
This means that $f \not\in \average{\Gamma_{\rm sub}}^2_{(\rho_2, \sigma_2)}$ by Lemma~\ref{lem:exp chara}.

\subsection{Proof of Theorem~\ref{thm:2aryksubExpressive}}\label{subsec:thm:2aryksubExpressive}
Let $f : [0,k]^2 \rightarrow \mathbf{Q}$ be a $k$-submodular function defined by
$f(0,0) := -1$, $f(1,1) := 1$, and $f(x) := 0$ for other $x$.
It suffices to prove $f \not\in \average{\Gamma_{\rm sub}}^k_{(\rho_k, \sigma_k)}$.

$\omega_k : {\rm Pol}^{(8)}(\average{\Gamma_{\rm sub}}^k_{(\rho_k, \sigma_k)}) \rightarrow \mathbf{Q}$ defined as the following is a weighted polymorphism of $\average{\Gamma_{\rm sub}}^k_{(\rho_k, \sigma_k)}$:
\begin{align*}
\omega_k(\varphi) := \begin{cases}
-5 & \text{if $\varphi \in \{e_1^{(8)}, e_2^{(8)}\}$},\\
-3 & \text{if $\varphi \in \{e_3^{(8)}, e_4^{(8)}, e_6^{(8)}, e_8^{(8)} \}$},\\
-2 & \text{if $\varphi \in \{e_5^{(8)}, e_7^{(8)}\}$},\\
3 & \text{if $\varphi \in \{\varphi_1, \varphi_4\}$},\\
2 & \text{if $\varphi \in \{\varphi_2, \varphi_3, \varphi_7, \varphi_8, \varphi_9, \varphi_{16} \}$},\\
1 & \text{if $\varphi \in \{\varphi_5, \varphi_6, \varphi_{10}, \varphi_{11}, \varphi_{12}, \varphi_{13}, \varphi_{14}, \varphi_{15} \}$},\\
0 & \text{otherwise}.
\end{cases}
\end{align*}
Here $\varphi_1, \dots, \varphi_{16} : [0,k]^8 \rightarrow [0,k]^8$ is defined by
\begin{align*}
&\varphi_1(a) := \sigma_3(\rho_3(b_1 \wedge b_4)),\\
&\varphi_2(a) := \sigma_3(\rho_3(b_1 \wedge b_7)),\\
&\varphi_3(a) := \sigma_3(\rho_3(b_2 \wedge b_5)),\\
&\varphi_4(a) := \sigma_3(\rho_3(b_2 \wedge b_8)),\\
&\varphi_5(a) := \sigma_3(\rho_3((b_1 \vee b_4) \wedge b_3)),\\
&\varphi_6(a) := \sigma_3(\rho_3((b_2 \vee b_8) \wedge b_6)),\\
&\varphi_7(a) := \sigma_3(\rho_3((b_1 \vee b_4) \wedge (b_2 \vee b_5) \wedge b_3)),\\
&\varphi_8(a) := \sigma_3(\rho_3((b_1 \vee b_7) \wedge (b_2 \vee b_8) \wedge b_6)),\\
&\varphi_9(a) := \sigma_3(\rho_3((b_1 \vee b_2 \vee b_4 \vee b_5) \wedge (b_1 \vee b_2 \vee b_7 \vee b_8))),\\
&\varphi_{10}(a) := \sigma_3(\rho_3((((b_1 \vee b_4) \wedge (b_2 \vee b_5)) \vee b_3) \wedge (b_2 \vee b_6 \vee b_8))),\\
&\varphi_{11}(a) := \sigma_3(\rho_3((((b_1 \vee b_7) \wedge (b_2 \vee b_8)) \vee b_6) \wedge (b_1 \vee b_3 \vee b_4))),\\
&\varphi_{12}(a) := \sigma_3(\rho_3(((b_1 \vee b_4) \wedge (b_2 \vee b_5) \vee b_3) \wedge (((b_1 \vee b_7) \wedge (b_2 \vee b_8)) \vee b_6))),\\
&\varphi_{13}(a) := \sigma_3(\rho_3(((b_1 \vee b_4) \wedge (b_2 \vee b_5)) \vee ((b_1 \vee b_7) \wedge (b_2 \vee b_8)) \vee b_3 \vee b_6)),\\
&\varphi_{14}(a) := \sigma_3(\rho_3((b_1 \vee b_2 \vee b_4 \vee b_5 \vee b_7 \vee b_8) \wedge (((b_1 \vee b_4) \wedge (b_2 \vee b_5)) \vee b_2 \vee b_3 \vee b_6 \vee b_8))),\\
&\varphi_{15}(a) := \sigma_3(\rho_3((b_1 \vee b_2 \vee b_4 \vee b_5 \vee b_7 \vee b_8) \wedge (((b_1 \vee b_7) \wedge (b_2 \vee b_8)) \vee b_1 \vee b_3 \vee b_4 \vee b_6))),\\
&\varphi_{16}(a) := \sigma_3(\rho_3(b_1 \vee b_2 \vee b_3 \vee b_4 \vee b_5 \vee b_6 \vee b_7 \vee b_8))\\
\end{align*}
for $a = (a_1, \dots,a_8) \in [0,k]^8$,
where $b_i := \sigma^{-1}(a_i) \in \{ \sigma_k(x) \mid x \in [0,k] \}$ for $i = 1,\dots,8$.
Indeed, a function $g \in \average{\Gamma_{\rm sub}}^k$ such that $g(v) \geq g(\rho_k(v))$ for $v \in {\rm dom}\ g$ satisfies the following inequalities for all $v_1,v_2, \dots, v_8 \in \{0,1\}^{kn}$;
\begin{align*}
&3g(v_1)+3g(v_4) \geq 3g(v_1 \wedge v_4) + 3g(v_1 \vee v_4),\\
&2g(v_1)+2g(v_7) \geq 2g(v_1 \wedge v_7) + 2g(v_1 \vee v_7),\\
&2g(v_2)+2g(v_5) \geq 2g(v_2 \wedge v_5) + 2g(v_2 \vee v_5),\\
&3g(v_2)+3g(v_8) \geq 3g(v_2 \wedge v_8) + 3g(v_2 \vee v_8),\\
&g(v_1 \vee v_4) + g(v_3) \geq g((v_1 \vee v_4) \wedge v_3) + g(v_1 \vee v_3 \vee v_4),\\
&g(v_2 \vee v_8) + g(v_6) \geq g((v_2 \vee v_8) \wedge v_6) + g(v_2 \vee v_6 \vee v_8),\\
&2g(v_1 \vee v_4) + 2g(v_2 \vee v_5) \geq 2g((v_1 \vee v_4) \wedge (v_2 \vee v_5)) + 2g(v_1 \vee v_2 \vee v_4 \vee v_5),\\
&2g(v_1 \vee v_7) + 2g(v_2 \vee v_8) \geq 2g((v_1 \vee v_7) \wedge (v_2 \vee v_8)) + 2g(v_1 \vee v_2 \vee v_7 \vee v_8),\\
&2g((v_1 \vee v_4) \wedge (v_2 \vee v_5)) + 2g(v_3) \geq 2g((v_1 \vee v_4) \wedge (v_2 \vee v_5) \wedge v_3) + 2g(((v_1 \vee v_4) \wedge (v_2 \vee v_5)) \vee v_3),\\
&2g((v_1 \vee v_7) \wedge (v_2 \vee v_8)) + 2g(v_6) \geq 2g((v_1 \vee v_7) \wedge (v_2 \vee v_8) \wedge v_6) + 2g(((v_1 \vee v_7) \wedge (v_2 \vee v_8)) \vee v_6),\\
&2g(v_1 \vee v_2 \vee v_4 \vee v_5) + 2g(v_1 \vee v_2 \vee v_7 \vee v_8)\\
&\quad \geq 2g((v_1 \vee v_2 \vee v_4 \vee v_5) \wedge (v_1 \vee v_2 \vee v_7 \vee v_8)) + 2g(v_1 \vee v_2 \vee v_4 \vee v_5 \vee v_7 \vee v_8),\\
&g(((v_1 \vee v_4) \wedge (v_2 \vee v_5)) \vee v_3) + g(v_2 \vee v_6 \vee v_8)\\
&\quad \geq g((((v_1 \vee v_4) \wedge (v_2 \vee v_5)) \vee v_3) \wedge (v_2 \vee v_6 \vee v_8)) + g((v_1 \vee v_4) \wedge (v_2 \vee v_5) \vee v_2 \vee v_3 \vee v_6 \vee v_8),\\
&g(((v_1 \vee v_7) \wedge (v_2 \vee v_8)) \vee v_6) + g(v_1 \vee v_3 \vee v_4)\\
&\quad \geq g((((v_1 \vee v_7) \wedge (v_2 \vee v_8)) \vee v_6) \wedge (v_1 \vee v_3 \vee v_4))+ g(((v_1 \vee v_7) \wedge (v_2 \vee v_8)) \vee v_1 \vee v_3 \vee v_4 \vee v_6),\\
&g(((v_1 \vee v_4) \wedge (v_2 \vee v_5)) \vee v_3) + g(((v_1 \vee v_7) \wedge (v_2 \vee v_8)) \vee v_6)\\
&\quad \geq g((((v_1 \vee v_4) \wedge (v_2 \vee v_5)) \vee v_3) \wedge (((v_1 \vee v_7) \wedge (v_2 \vee v_8)) \vee v_6))\\
&\qquad + g(((v_1 \vee v_4) \wedge (v_2 \vee v_5)) \wedge ((v_1 \vee v_7) \wedge (v_2 \vee v_8)) \vee v_3 \vee v_6),\\
&g(v_1 \vee v_2 \vee v_4 \vee v_5 \vee v_7 \vee v_8) + g((v_1 \vee v_4) \wedge (v_2 \vee v_5) \vee v_2 \vee v_3 \vee v_6 \vee v_8)\\
&\quad \geq g((v_1 \vee v_2 \vee v_4 \vee v_5 \vee v_7 \vee v_8) \wedge ((v_1 \vee v_4) \wedge (v_2 \vee v_5) \vee v_2 \vee v_3 \vee v_6 \vee v_8))\\
&\qquad + g(v_1 \vee v_2 \vee v_3 \vee v_4 \vee v_5 \vee v_6 \vee v_7 \vee v_8),\\
&g(v_1 \vee v_2 \vee v_4 \vee v_5 \vee v_7 \vee v_8) + g(((v_1 \vee v_7) \wedge (v_2 \vee v_8)) \vee v_1 \vee v_3 \vee v_4 \vee v_6)\\
&\quad \geq g((v_1 \vee v_2 \vee v_4 \vee v_5 \vee v_7 \vee v_8) \wedge (((v_1 \vee v_7) \wedge (v_2 \vee v_8)) \vee v_1 \vee v_3 \vee v_4 \vee v_6))\\
&\qquad + g(v_1 \vee v_2 \vee v_3 \vee v_4 \vee v_5 \vee v_6 \vee v_7 \vee v_8),\\
&3g(v_1 \wedge v_4) \geq 3g(\rho_3(v_1 \wedge v_4)),\\
&2g(v_1 \wedge v_7) \geq 2g(\rho_3(v_1 \wedge v_7)),\\
&2g(v_2 \wedge v_5) \geq 2g(\rho_3(v_2 \wedge v_5)),\\
&3g(v_2 \wedge v_8) \geq 3g(\rho_3(v_2 \wedge v_8)),\\
&g((v_1 \vee v_4) \wedge v_3) \geq g(\rho_3((v_1 \vee v_4) \wedge v_3)),\\
&g((v_2 \vee v_8) \wedge v_6) \geq g(\rho_3((v_2 \vee v_8) \wedge v_6)),\\
&2g((v_1 \vee v_4) \wedge (v_2 \vee v_5) \wedge v_3) \geq 2g(\rho_3((v_1 \vee v_4) \wedge (v_2 \vee v_5) \wedge v_3)),\\
&2g((v_1 \vee v_7) \wedge (v_2 \vee v_8) \wedge v_6) \geq 2g(\rho_3((v_1 \vee v_7) \wedge (v_2 \vee v_8) \wedge v_6)),\\
&2g((v_1 \vee v_2 \vee v_4 \vee v_5) \wedge (v_1 \vee v_2 \vee v_7 \vee v_8)) \geq 2g(\rho_3((v_1 \vee v_2 \vee v_4 \vee v_5) \wedge (v_1 \vee v_2 \vee v_7 \vee v_8))),\\
&g((((v_1 \vee v_4) \wedge (v_2 \vee v_5)) \vee v_3) \wedge (v_2 \vee v_6 \vee v_8)) \geq g(\rho_3((((v_1 \vee v_4) \wedge (v_2 \vee v_5)) \vee v_3) \wedge (v_2 \vee v_6 \vee v_8))),\\
&g((((v_1 \vee v_7) \wedge (v_2 \vee v_8)) \vee v_6) \wedge (v_1 \vee v_3 \vee v_4)) \geq g(\rho_3((((v_1 \vee v_7) \wedge (v_2 \vee v_8)) \vee v_6) \wedge (v_1 \vee v_3 \vee v_4))),\\
&g((((v_1 \vee v_4) \wedge (v_2 \vee v_5)) \vee v_3) \wedge (((v_1 \vee v_7) \wedge (v_2 \vee v_8)) \vee v_6))\\
&\quad \geq g(\rho_3(((v_1 \vee v_4) \wedge (v_2 \vee v_5) \wedge v_3) \wedge (((v_1 \vee v_7) \wedge (v_2 \vee v_8)) \vee v_6))),\\
&g(((v_1 \vee v_4) \wedge (v_2 \vee v_5)) \wedge ((v_1 \vee v_7) \wedge (v_2 \vee v_8)) \vee v_3 \vee v_6)\\
&\quad \geq g(\rho_3(((v_1 \vee v_4) \wedge (v_2 \vee v_5)) \wedge ((v_1 \vee v_7) \wedge (v_2 \vee v_8)) \vee v_3 \vee v_6)),\\
&g((v_1 \vee v_2 \vee v_4 \vee v_5 \vee v_7 \vee v_8) \wedge ((v_1 \vee v_4) \wedge (v_2 \vee v_5) \vee v_2 \vee v_3 \vee v_6 \vee v_8))\\
&\quad \geq g(\rho_3((v_1 \vee v_2 \vee v_4 \vee v_5 \vee v_7 \vee v_8) \wedge ((v_1 \vee v_4) \wedge (v_2 \vee v_5) \vee v_2 \vee v_3 \vee v_6 \vee v_8))),\\
&g((v_1 \vee v_2 \vee v_4 \vee v_5 \vee v_7 \vee v_8) \wedge (((v_1 \vee v_7) \wedge (v_2 \vee v_8)) \vee v_1 \vee v_3 \vee v_4 \vee v_6))\\
&\quad \geq g(\rho_3((v_1 \vee v_2 \vee v_4 \vee v_5 \vee v_7 \vee v_8) \wedge (((v_1 \vee v_7) \wedge (v_2 \vee v_8)) \vee v_1 \vee v_3 \vee v_4 \vee v_6)),\\
&2g(v_1 \vee v_2 \vee v_3 \vee v_4 \vee v_5 \vee v_6 \vee v_7 \vee v_8) \geq 2g(\rho_3(v_1 \vee v_2 \vee v_3 \vee v_4 \vee v_5 \vee v_6 \vee v_7 \vee v_8)).
\end{align*}
Also $g$ satisfies the inequality given by summing up the all above inequalities.
Hence if it holds that $f \in \average{\Gamma_{\rm sub}}^k_{(\rho_k, \sigma_k)}$,
then $f$ satisfies
\begin{align}\label{eq:ksub}
\sum_{\varphi \in {\rm Pol}^{(8)}(\average{\Gamma_{\rm sub}}^k_{(\rho_k, \sigma_k)})} \omega(\varphi) f(\varphi(x^1, \dots,x^8)) \leq 0
\end{align}
for every $x^1, \dots,x^8 \in \textrm{dom }f$.

Let $x^1,\dots,x^8 \in [0,k]^2$ be defined by
$x^1 := (1, 2)$, $x^2 := (1, 3)$, $x^3 := (2, 1)$, $x^4 := (2, 2)$,
$x^5 := (2, 3)$, $x^6 := (3, 1)$, $x^7 := (3, 2)$, and $x^8 := (3, 3)$.
Since $f(x^i) = 0$ for $i = 1,\dots, 8$, if $\varphi \in \{ e_1^{(8)}, \dots, e_8^{(8)} \}$
then $\omega(\varphi) f(\varphi(x^1, \dots, x^8)) = 0$.
We can see that $f$ does not satisfy (\ref{eq:ksub}) by Table~\ref{tab:ksub}.
Therefore we obtain $f \not\in \average{\Gamma_{\rm sub}}^k_{(\rho_k, \sigma_k)}$.
\begin{table}[htb]
\caption{Calculations of $\omega(\varphi_i) f(\varphi_i(x^1, \dots,x^8))$ for $i = 1, \dots, 16$.}
\begin{center}
\begin{tabular}{c|cccc}
$i$ & $\varphi_i(x^1, \dots, x^8)$ & $\omega(\varphi_i) f(\varphi_i(x^1,\dots,x^8))$ \\ \hline
1 & (0,2) & 0\\
2 & (0,2) & 0\\
3 & (0,3) & 0\\
4 & (0,3) & 0\\
5 & (2,0) & 0\\
6 & (3,0) & 0\\
7 & (2,0) & 0\\
8 & (3,0) & 0\\
9 & (1,0) & 0\\
10 & (1,1) & 1\\
11 & (1,1) & 1\\
12 & (1,1) & 1\\
13 & (0,1) & 0\\
14 & (0,3) & 0\\
15 & (0,2) & 0\\
16 & (0,0) & $-2$\\
\end{tabular}
\end{center}
\label{tab:ksub}
\end{table}

\section*{Acknowledgments}
We thank Hiroshi Hirai and Magnus Wahlstr{\"o}m for careful reading and helpful comments.
We also thank the referees for pointing out a similarity between extended expressive power and pp-interpretation or weighted variety and for information on the papers~\cite{ICALP/KO15, SICOMP/TZ17}.
This research is supported by JSPS Research Fellowship for Young Scientists and by JST ERATO Grant Number JPMJER1305, Japan.


\end{document}